\documentclass[english,round]{article}
\usepackage[T1]{fontenc}
\usepackage[latin9]{inputenc}
\usepackage{geometry}
\usepackage{color}
\usepackage{babel}
\usepackage{verbatim}
\usepackage{units}
\usepackage{mathtools}
\usepackage{enumitem}
\usepackage{amsmath}
\usepackage{amsthm}
\usepackage{amssymb}
\usepackage[authoryear]{natbib}
\usepackage[unicode=true,
 bookmarks=true,bookmarksnumbered=false,bookmarksopen=false,
 breaklinks=false,pdfborder={0 0 0},pdfborderstyle={},backref=page,colorlinks=true]
 {hyperref}
\hypersetup{pdftitle={Entropy Based Risk Measures},
 pdfauthor={Alois Pichler, Ruben Schlotter},
 citecolor=blue, urlcolor= magenta, linkcolor= blue}

\makeatletter
\theoremstyle{plain}
\newtheorem{thm}{\protect\theoremname}
  \theoremstyle{definition}
  \newtheorem{defn}[thm]{\protect\definitionname}
  \theoremstyle{remark}
  \newtheorem{rem}[thm]{\protect\remarkname}
  \theoremstyle{plain}
  \newtheorem{lem}[thm]{\protect\lemmaname}
  \theoremstyle{plain}
  \newtheorem{prop}[thm]{\protect\propositionname}
  \theoremstyle{plain}
  \newtheorem{cor}[thm]{\protect\corollaryname}

\usepackage{dsfont}   
\usepackage{newtxtext, newtxmath}	

\setlist[enumerate,1]{label=(\roman*)}  
\usepackage{microtype}          
\newcommand{\sign}{{\operatorname{sign}}}
\DeclareMathOperator{\E}{{\mathbb E}}
\DeclareMathOperator{\EVaR}{\mathsf {EV@R}}
\DeclareMathOperator{\AVaR}{\mathsf {AV@R}}

\DeclareMathOperator{\VaR}{\mathsf {V@R}}
\DeclareMathOperator{\one}{{\mathds 1}} 		
\DeclareMathOperator{\esssup}{\operatornamewithlimits{ess\,sup}}

\DeclareMathOperator*{\argmin}{arg\,min}

\usepackage{doi}

\makeatother

  \providecommand{\corollaryname}{Corollary}
  \providecommand{\definitionname}{Definition}
  \providecommand{\lemmaname}{Lemma}
  \providecommand{\propositionname}{Proposition}
  \providecommand{\remarkname}{Remark}
\providecommand{\theoremname}{Theorem}

\begin{document}

\title{Entropy Based Risk Measures }

\author{Alois Pichler\thanks{Both authors: Technische Universität Chemnitz, Fakultät für Mathematik.
90126 Chemnitz, Germany.}  \and  Ruben Schlotter\thanks{Corresponding author. Contact: \protect\href{mailto:ruben.schlotter@math.tu-chemnitz.de}{ruben.schlotter@math.tu-chemnitz.de}}}
\maketitle
\begin{abstract}
Entropy is a measure of self-information which is used to quantify
information losses. Entropy was developed in thermodynamics, but is
also used to compare probabilities based on their deviating information
content. Corresponding model uncertainty is of particular interest
and importance in stochastic programming and its applications like
mathematical finance, as complete information is not accessible or
manageable in general.

This paper extends and generalizes the Entropic Value-at-Risk by involving
Rényi entropies. We provide explicit relations of different entropic
risk measures, we elaborate their dual representations and present
their relations explicitly. 

We consider the largest spaces which allow studying the impact of
information in detail and it is demonstrated that these do not depend
on the information loss. The dual norms and Hahn\textendash Banach
functionals are characterized explicitly. 


\textbf{Classification:} 90C15, 60B05, 62P05
\end{abstract}

\section{Introduction}

\citet{boltzmann1877-R} defines the entropy of a thermodynamic system
as a measure of how many different microstates could give rise to
the macrostate the system is in. He gives the famous formula
\[
{\displaystyle S=k\log W}
\]
for the entropy of the thermodynamic system, where $S$ is the entropy
of the macrostate, $k$ is Boltzmann's constant and $W$ is the total
number of possible microstates that might yield the macrostate. It
then becomes natural to interpret entropy as a measure of disorder.
\citet{Shannon1948-R} defines the \emph{information entropy} of a
discrete random variable $Z$ with probability mass function $f(\cdot)$
by 
\[
H(Z)=\sum_{x}f(x)\log f(x),
\]
which  extends to
\[
H(Z)=\E Z\log Z
\]
in the continuous case. Information entropy is interpreted as the
average amount of information produced by the probabilistic source
of data~$Z$. Relating this to Boltzmann's entropy formula one can
say that information entropy of a system is the amount of information
needed to determine a microstate, given the macrostate. Many extensions
of information entropy (now often called Shannon entropy) have been
introduced. The most notable extensions are \emph{Rényi entropies
of order $q$, }specified as 
\[
H_{q}(Z):=\frac{1}{q-1}\log\E Z^{q}.
\]

Related to Shannon entropy is the quantity
\[
D\left(Q\|P\right)=\sum_{x}f\left(x\right)\log\frac{f\left(x\right)}{g(x)}
\]
called \emph{relative entropy} or \emph{Kullback\textendash Leibler
divergence}. Here, $f\left(\cdot\right)$ ($g\left(\cdot\right)$,
resp.) is the probability mass functions of the probability measure
$P$ ($Q$, resp.). Relative entropy describes the information loss
when considering the distribution $g$ while $f$ is the correct distribution.
Divergence can therefore be seen as a distance of probability measures,
although it is not a metric since it is neither symmetric nor does
it satisfy the triangle inequality. Divergences can be derived from
different entropies in analogy to relative entropy. For Rényi entropy
we obtain Rényi divergences given by 
\begin{align*}
D_{R}^{q}\left(Q\|P\right) & =\frac{1}{q-1}\log\E Z^{q}=H_{q}(Z),
\end{align*}
where $Z=\frac{\mathrm{d}Q}{\mathrm{d}P}$ is the Radon\textendash Nikodým
derivative. The family of Rényi divergences is related to Hellinger
divergences defined as
\[
D_{T}^{q}(Q\|P)=\frac{1}{q-1}\E\left(Z^{q}-1\right)
\]
(see \citet{Liese2006-R}), where $Z$ is as above. Hellinger divergence
is sometimes called Tsallis divergence.

For an overview of entropy in information theory we may refer to \citet{Cover2006-R}.
For the relationship between different divergences see \citet{amari2010-R},
\citet{Amari2009-R} and \citet{Liese2006-R}. For information specific
to Rényi divergence we refer the interested reader to \citet{vanErven2014-R}. 

Risk measures play an important role in finance, stochastic optimization,
e.g. In fact, in risk-averse stochastic optimization one is often
interested in problems of the form
\begin{align}
\text{minimize } & \rho\ \big(c(x,Y)\big)\label{eq:stochOpt}\\
\text{subject to } & x\in\mathbb{X},\nonumber 
\end{align}
where $\rho$ is a risk measure, $c$ is a cost function of a random
variable $Y$ and $\mathbb{X}$ is a decision space. In this paper
we focus our attention on risk measures based on entropy. In particular,
we address generalizations of the Entropic Value-at-Risk $(\EVaR),$
a coherent risk measure introduced in~\citet{AhmadiJavidEVaR-R,AhmadiJavidEVaR2-R}.
It is known that $\EVaR$ can be represented using relative entropy
or Kullback\textendash Leibler divergence as 
\begin{equation}
\EVaR_{\alpha}(Y)=\sup\left\{ \E YZ\colon Z\geq0,\,\E Z=1,\,\E Z\log Z\leq\log\frac{1}{1-\alpha}\right\} ,\label{eq:EVaRsup}
\end{equation}
where $Z$ is a density with respect to the reference probability
measure $P$. This risk measure corresponds to the worst expected
value with respect to probability measures with information content
not deviating by more than $\log\frac{1}{1-\alpha}$ from the baseline
distribution $P$. 

\medskip{}

In this paper we generalize the risk measure~\eqref{eq:EVaRsup}
and consider new risk measures by replacing the relative entropy in
the dual representation~\eqref{eq:EVaRsup} with different divergences
as suggested in \citet{AhmadiJavidEVaR-R} first. \citet{Breuer2013-R,Breuer2013a-R}
study the class of $\phi$-entropic risk measures (cf.\ also \citet{Bellini2008986-R}).
\citet{KovacevicBreuer2016-R} study a multiperiod extension of~\eqref{eq:EVaRsup}
while \citet{Foellmer2011-R} study a version of an entropic risk
measure too, but in a very different conceptual setting.

Extending this direction of research involving information losses
motivates studying risk measures of the form
\begin{equation}
\rho(Y)=\sup\left\{ \E YZ\colon Z\geq0,\,\E Z=1,\,H_{q}(Z)\leq\log\frac{1}{1-\alpha}\right\} ,\label{eq:RVaRsup}
\end{equation}
where $H_{q}$ is the Rényi entropy. For $q>1$ these risk measures
are related to \emph{higher order dual risk measures} considered in
\citet{Dentcheva2010-R}. Incorporating entropy into the definition
of the risk measure allows a consistent interpretation based on information
theory. 

The new class of risk measures based on Rényi divergence recovers
well-known coherent risk measures such as the Average Value-at-risk
($\AVaR)$, the expectation and the essential supremum, the classical
Entropic Value-at-Risk is a special case as well. In fact, our new
class of entropy based risk measures interpolates the Average Value-at-Risk
and the classical Entropic Value-at-Risk. This provides a flexible
class of risk measures for modeling stochastic optimization problems
taking information into account.

Returning to the standard problem~\eqref{eq:stochOpt} in risk-averse
stochastic optimization we notice that for risk measures of the form~\eqref{eq:RVaRsup}
the problem~\eqref{eq:stochOpt} becomes a minimax problem which
in general is difficult. However, extending the work of \citet{Dentcheva2010-R}
we provide an equivalent infimum representation of~\eqref{eq:RVaRsup}
which facilitates the problem.

We further study the norms associated with the risk measure $\rho$
defined as 
\begin{equation}
\left\Vert Y\right\Vert _{\rho}\coloneqq\rho\left(\left|Y\right|\right)\label{eq:7-3}
\end{equation}
and their corresponding dual norms. We give explicit characterizations
for the associated Hahn\textendash Banach functionals.

\paragraph{Mathematical Setting.}

We consider a vector space $L$ of $\mathbb{R}$-valued random variables
on a reference probability space $\left(\Omega,\,\mathcal{F},\,P\right)$.
The set $L$ is called $\textit{model space},$ which is used in this
paper to represent a set of random losses. This setting is typical
in stochastic optimization as well as in the insurance literature,
while in a finance context one is primarily interested in monetary
gains instead of losses. Throughout this paper we work with random
variables in $L^{p}(\Omega,\,\mathcal{F},\,P)$ for $p\geq1$ or $L^{\infty}(\Omega,\,\mathcal{F},\,P)$.
We shall call exponents $p$ and $q$ \emph{conjugate} (\emph{Hölder
conjugate}, resp.), if $\nicefrac{1}{p}+\nicefrac{1}{q}=1$, even
for $p<1$ or $q<1$. Typically, we denote the exponent conjugate
to $p$ by $p^{\prime}$. As usual, we set $p^{\prime}=\infty$ for
$p=1$ and $p^{\prime}=1$ for $p=\infty$. 

\medskip{}

A risk measure $\rho\colon L\to\mathbb{R}\cup\left\{ \infty\right\} $
is called $\textit{coherent}$ if it satisfies the following four
properties introduced by \citet{Artzner1999-R}.
\begin{enumerate}[noitemsep]
\item Translation equivariance: $\rho\left(Y+c\right)=\rho(Y)+c$ for any
$Y\in L$ and $c\in\mathbb{R}$;
\item Subadditivity: $\rho\left(Y_{1}+Y_{2}\right)\leq\rho\left(Y_{1}\right)+\rho\left(Y_{2}\right)$
for all $Y_{1},\,Y_{2}\in L$;
\item Monotonicity: if $Y_{1},\,Y_{2}\in L$ and $Y_{1}\leq Y_{2},$ then
$\rho\left(Y_{1}\right)\leq\rho\left(Y_{2}\right)$;
\item Positive homogeneity: $\rho\left(\lambda Y\right)=\lambda\rho(Y)$
for all $Y\in L$ and $\lambda>0$.
\end{enumerate}
\citet{Delbaen02coherentrisk-R} gives a general representation of
coherent risk measures in the form 
\begin{equation}
\rho(Y)=\sup\left\{ \E_{Q}Y\colon Q\in\Gamma\right\} ,\label{eq:DelbaenRep}
\end{equation}
where $\Gamma$ is a set of probability measures satisfying certain
regularity conditions. Risk measures considered in~\eqref{eq:RVaRsup}
are thus coherent. 

\paragraph{Outline of the paper.}

\noindent Section~\ref{sec:Entropy} introduces and discusses important
properties of the Rényi entropy, Section~\ref{sec:RM} then introduces
the new risk measure $\EVaR$ and discusses its domain. In Section~\ref{sec:Duality}
we give the dual representation while Section~\ref{sec:Analytic-prop}
relates $\EVaR$-norms with the norm in $L^{p}$. This section discusses
the relation of different levels and orders as well. Section~\ref{sec:Dual-Norms}
deals with the dual norms and the Kusuoka representation, Section~\ref{sec:Summary}
concludes.

\section{\label{sec:Entropy}Entropy}

This section defines the Rényi entropy and features important properties
which we relate to in later sections.
\begin{defn}
We shall call a random variable $Z\in L^{1}$ with $Z\ge0$ and $\E Z=1$
a \emph{density} with respect to $P$, or simply density.
\end{defn}

\begin{defn}
\emph{\label{def:Renyi} }The \emph{Rényi entropy}\footnote{Named after Alfréd Rényi, 1921\textendash 1970, Hungarian mathematician}
of order $q\in\mathbb{R}$ of a density $Z$ is\footnote{We employ the analytic continuation of the mapping $z\mapsto z\cdot\log z$
by setting $0\log0:=0$.}
\begin{equation}
H_{q}(Z):=\begin{cases}
-\log P(Z>0) & \text{if }q=0,\\
\E Z\log Z & \text{if }q=1,\\
\log\left\Vert Z\right\Vert _{\infty} & \text{if }q=\infty,\\
\frac{1}{q-1}\log\E Z^{q} & \text{else},
\end{cases}\label{eq:DefRenyi}
\end{equation}
provided that the expectations are finite (note that one has to assume
$Z>0$ in order to have $H_{q}(Z)$ well-defined for $q<0$).
\end{defn}

Rényi entropy, as introduced in~\eqref{eq:DefRenyi}, is continuous
in $q\in\mathbb{R}$. Indeed, by l'Hôpital's rule we have 
\begin{equation}
\lim_{q\to1}H_{q}(Z)=\lim_{q\to1}\frac{\E Z^{q}\log Z}{\E Z^{q}}=\E Z\log Z=H_{1}(Z),\label{eq:H1}
\end{equation}
so that the entropy of order $q=1$ in~\eqref{eq:DefRenyi} is the
continuous extension of $\frac{1}{q-1}\log\E Z^{q}$. Furthermore,
for $q\to\infty$, it holds that 
\[
\lim_{q\to\infty}H_{q}(Z)=\lim_{q\to\infty}\frac{q}{q-1}\log\left\Vert Z\right\Vert _{q}=\log\left\Vert Z\right\Vert _{\infty}.
\]
For $q\to0$ we get 
\[
\lim_{q\to0}H_{q}(Z)=\lim_{q\to0}\frac{1}{q-1}\log\int_{\Omega}\one_{Z>0}Z^{q}\,\mathrm{d}P=-\log\,P\left\{ Z>0\right\} 
\]
and hence the case $q=0$ in Definition~\ref{def:Renyi} is consistent
as well.

\begin{rem}
\label{rem: norm-ent}The expression
\[
\left\Vert Z\right\Vert _{q}=\left(\E Z^{q}\right)^{\frac{1}{q}}
\]
is not a norm whenever $q<1$, but we will employ it to allow for
a compact notation. With this notation at hand the Rényi entropy rewrites
as 
\begin{equation}
H_{q}(Z)=\frac{q}{q-1}\log\left\Vert Z\right\Vert _{q}.\label{eq:Zq}
\end{equation}
\end{rem}

\subsection*{Properties of Rényi Entropy}

The entropy $H_{q}(Z)$ is nonnegative for $q\geq1$ as we have that
$\left\Vert Z\right\Vert _{q}\geq\E Z=1$. For $0<q<1$ the exponents
$\frac{1}{q}$ and $\frac{1}{1-q}$ are conjugate so that by Hölder's
inequality $\E Z^{q}\le\left\Vert Z^{q}\right\Vert _{\frac{1}{q}}\cdot\left\Vert \one\right\Vert _{\frac{1}{1-q}}=\E Z=1$
and consequently $H_{q}(Z)=\frac{1}{q-1}\log\E Z^{q}\geq0$ even for
$q\in(0,1)$. Together with the special case $q=0$ we thus have that
the entropy is nonnegative for all $q\ge0$, 
\[
H_{q}(Z)\ge0\qquad(q\ge0).
\]

The elementary relation $P(Z>0)\cdot\left\Vert Z\right\Vert _{\infty}\ge1$
follows from $\E Z=1$ and consequently we have that $H_{0}(Z)\le H_{\infty}(Z)$.
The next lemma reveals the general monotonic behavior of the Rényi
entropy in its order $q$.
\begin{lem}
\label{lem:R=0000E9nyi-Entropy-Mono} The Rényi entropy $H_{q}(Z)$
is non-decreasing in its order $q$ for every $Z$ fixed. Further,
there exists a density $Z$ with arbitrary entropy so that $H_{q}(Z)$
is constant for $q\ge0$.
\end{lem}

\begin{proof}
The derivative of~\eqref{eq:DefRenyi}  with respect to the order
$q$ is 
\begin{align*}
\frac{\mathrm{d}}{\mathrm{d}q}H_{q}(Z) & =-\frac{1}{\left(q-1\right){}^{2}}\log\E Z^{q}+\frac{1}{q-1}\frac{\E Z^{q}\log Z}{\E Z^{q}},
\end{align*}
which can be restated as 
\begin{align}
\frac{\mathrm{d}}{\mathrm{d}q}H_{q}(Z) & =-\frac{1}{\left(q-1\right){}^{2}}\log\E Z^{q}+\frac{q-1}{\left(q-1\right){}^{2}}\E\frac{Z^{q}}{\E Z^{q}}\log Z\nonumber \\
 & =\frac{1}{\left(q-1\right){}^{2}}\E\frac{Z^{q}}{\E Z^{q}}\log\frac{1}{Z}\frac{Z^{q}}{\E Z^{q}}=\frac{1}{\left(q-1\right){}^{2}}\E Z_{q}\log\frac{Z_{q}}{Z},\label{eq:7}
\end{align}
where we employ the abbreviation $Z_{q}$ for the power-density $Z_{q}:=\frac{Z^{q}}{\E Z^{q}}$.
In line with the proof that the Kullback\textendash Leibler divergence
is non-negative we consider the Bregman divergence 
\[
D(y,z):=\varphi(y)-\varphi(z)-\varphi^{\prime}(z)\cdot\left(y-z\right).
\]
For the convex function $\varphi(z):=z\log z$ (with derivative $\varphi^{\prime}(z)=1+\log z$)
we get by convexity
\[
0\le D\left(Z_{q},Z\right)=Z_{q}\log Z_{q}-Z\log Z-\left(1+\log Z\right)\cdot\left(Z_{q}-Z\right).
\]
Taking expectations and expanding gives 
\[
\E Z_{q}\log\frac{Z_{q}}{Z}=\E\left[Z_{q}\log Z_{q}-Z\log Z-Z_{q}\log Z+Z\log Z\right]\ge0.
\]
It follows from~\eqref{eq:7} that $\frac{\mathrm{d}}{\mathrm{d}q}H_{q}(Z)\ge0$
and thus the assertion.

Now consider the random variable $Z$ with $\alpha\in\left(0,1\right)$
and $P\left(Z=0\right)=\alpha\text{ and }P\left(Z=\frac{1}{1-\alpha}\right)=1-\alpha.$
The random variable $Z$ is a density with entropy 
\[
H_{q}(Z)=\frac{1}{q-1}\log\frac{1-\alpha}{\left(1-\alpha\right){}^{q}}=\log\frac{1}{1-\alpha},
\]
which is independent of the order $q$. 
\end{proof}
\begin{rem}
\label{rem:Pos-Entr}From Remark~\ref{rem: norm-ent} it is clear
that we have $H_{0}(Z)=0$ for $Z>0$ and hence Lemma~\ref{lem:R=0000E9nyi-Entropy-Mono}
implies that 
\[
H_{q}(Z)\leq0
\]
for $q<0$.
\end{rem}

We state convexity properties of the \emph{Rényi} entropy for varying
order $q$ next. 
\begin{prop}[Convexity]
\label{prop:Convexity}The mapping $q\mapsto\left(q-1\right)\cdot H_{q}(Z)$
is a convex function on $\mathbb{R}$.
\end{prop}

\begin{proof}
For $\lambda\in\left(0,1\right)$ and $q_{0}$, $q_{1}\in\mathbb{R}$
define $q_{\lambda}:=\left(1-\lambda\right)q_{0}+\lambda q_{1}$.
By Hölder's inequality we have for arbitrary $q$ (such that every
integral exists)
\begin{equation}
\log\E Z^{q_{\lambda}}=\log\E\left(Z^{q}\cdot Z^{q_{\lambda}-q}\right)\le\frac{1}{p}\log\left(\E Z^{qp}\right)+\frac{1}{p^{\prime}}\log\left(\E Z^{\left(q_{\lambda}-q\right)p^{\prime}}\right),\label{eq:5-2}
\end{equation}
where $p^{\prime}=\frac{p}{p-1}$ is Hölder's conjugate exponent to
$p$. 

Choose $p:=\frac{1}{1-\lambda}$ and $q:=\left(1-\lambda\right)q_{0}$
and observe that $qp=q_{0}$, $p^{\prime}=\frac{1}{\lambda}$ and
$\left(q_{\lambda}-q\right)p^{\prime}=q_{1}$. The inequality~\eqref{eq:5-2}
thus reads
\[
\log\E Z^{q_{\lambda}}\le\left(1-\lambda\right)\log\left(\E Z^{q_{0}}\right)+\lambda\log\left(\E Z^{q_{1}}\right),
\]
from which the assertion follows.
\end{proof}
The preceding Proposition~\ref{prop:Convexity} extends to the case
$q=\infty$ in the following way.
\begin{prop}
\label{prop: Hqconvex}For $q,\tilde{q}\in\mathbb{R}$ and $q<\tilde{q}$
it holds that 
\[
\left(\tilde{q}-1\right)H_{\tilde{q}}(Z)\le\left(q-1\right)H_{q}(Z)+\left(\tilde{q}-q\right)H_{\infty}(Z)
\]
whenever the integrals are well defined. 
\end{prop}

\begin{proof}
Again by Hölder's inequality we have for $q<\tilde{q}$ that $\E Z^{\tilde{q}}\le\E\left(Z^{q}\cdot\left\Vert Z\right\Vert _{\infty}^{\tilde{q}-q}\right).$
Thus 
\[
\log\E Z^{\tilde{q}}\le\log\E Z^{q}+\left(\tilde{q}-q\right)\log\left\Vert Z\right\Vert _{\infty},
\]
i.e., 
\[
\left(\tilde{q}-1\right)H_{\tilde{q}}(Z)\le\left(q-1\right)H_{q}(Z)+\left(\tilde{q}-q\right)H_{\infty}(Z),
\]
which is the assertion.
\end{proof}

\section{\label{sec:RM}Risk measures based on Rényi entropy}

We now define entropic risk measures based on Rényi entropy. We start
from the dual representation of coherent risk measures first introduced
in~\citet{Delbaen02coherentrisk-R}. The constant $\log\frac{1}{1-\alpha}$
in the definition below is chosen to relate the entropic risk measures
to the Average Value-at-Risk and to the Value-at-Risk with confidence
level $\alpha.$
\begin{defn}[Risk measures based on Rényi entropy]
\label{def:Entropic}The Entropic Value-at-Risk $\EVaR_{\alpha}^{p}$
of order $p\in\mathbb{R}$ at confidence level $\alpha\in\left[0,1\right)$
and $Y\in L^{p}$ based on Rényi entropy is 
\begin{equation}
\EVaR_{\alpha}^{p}(Y):=\sup\left\{ \E YZ\colon Z\ge0,\,\E Z=1\text{ and }H_{p^{\prime}}(Z)\le\log\frac{1}{1-\alpha}\right\} ,\label{eq: EVaRDef1}
\end{equation}
where $\frac{1}{p}+\frac{1}{p^{\prime}}=1.$ For $p=\infty$ we set
$\EVaR_{\alpha}^{\infty}(Y)\coloneqq\EVaR_{\alpha}(Y)$ (cf.~\eqref{eq:H1}),
i.e., 
\begin{equation}
\EVaR_{\alpha}(Y):=\sup\left\{ \E YZ\colon Z\ge0,\,\E Z=1,\,\E Z\log Z\le\log\frac{1}{1-\alpha}\right\} \label{eq: EVaRDef2}
\end{equation}
and for $p=1$
\[
\EVaR_{\alpha}^{1}(Y)=\AVaR_{\alpha}(Y):=\sup\left\{ \E YZ\colon Z\geq0,\,\E Z=1,\,Z\leq\frac{1}{1-\alpha}\right\} .
\]
For $\alpha=1$ we set $\EVaR_{1}^{p}(Y):=\esssup Y.$
\end{defn}

\begin{rem}[The confidence level $\alpha=0$]
The Entropic Value-at-Risk based on Rényi entropy is nondecreasing
in $\alpha$, as $\alpha\mapsto\log\frac{1}{1-\alpha}$ is an increasing
function. Also note that $\log\frac{1}{1-\alpha}=0$ whenever $\alpha=0$,
and $H_{p^{\prime}}(Z)=0$ if and only if $Z=\one$ whenever $p^{\prime}\ge0$.
Hence, $\EVaR_{0}^{p}(Y)=\E Y$ for $p\not\in(0,1)$. Theorem~\ref{thm:RVaR(p<1)}
below addresses the case $p\in(0,1)$. 
\end{rem}

For $p>1$, the risk measure $\EVaR_{\alpha}^{p}\left(\cdot\right)$
is well defined on $L^{p}$ since
\begin{equation}
\EVaR_{\alpha}^{p}(Y)\leq\left\Vert Y\right\Vert _{p}\left(\frac{1}{1-\alpha}\right)^{\frac{1}{p}}.\label{eq:7-1}
\end{equation}
An in-depth discussion of this case can be found in \citet{Dentcheva2010-R}.

\medskip{}

This paper particularly extends the Entropic Value-at-Risk for $p<1$.
 To this end it is useful to revise the Hölder and Minkowski inequality
for $p<1.$ Since the inequalities in both cases are reversed, they
are sometimes called \emph{reverse Hölder} and \emph{reverse Minkowski
inequality}, respectively.
\begin{lem}[Reverse Hölder and reverse Minkowski inequality]
For $p\in(0,1)$ and $Y,\,Z\in L^{\infty}$ with $Z>0$ the inequality
\begin{equation}
\left\Vert YZ\right\Vert _{1}\geq\left\Vert Y\right\Vert _{p}\left\Vert Z\right\Vert _{q}\label{eq:5}
\end{equation}
holds true, where $q=\frac{p}{p-1}<0$ is the Hölder exponent conjugate
to $p$. 

For $Z_{1},\,Z_{2}\in L^{\infty}$ such that $Z_{1},\,Z_{2}>0$ and
$q<1$ we have
\[
\left\Vert Z_{1}+Z_{2}\right\Vert _{q}\geq\left\Vert Z_{1}\right\Vert _{q}+\left\Vert Z_{2}\right\Vert _{q}.
\]
\end{lem}

\begin{proof}
Without loss of generality we may rescale $Y$ and $Z$ such that
$\left\Vert YZ\right\Vert _{1}=\left\Vert Z\right\Vert _{q}=1.$ Then
the desired inequality~\eqref{eq:5} reduces to $1\geq\left\Vert Y\right\Vert _{p}^{p}=\left\Vert \left(\frac{\left|Y\right|Z}{Z}\right)^{p}\right\Vert _{1}.$
To accept the latter apply Hölder's inequality to $\left(\frac{\left|Y\right|Z}{Z}\right)^{p}$
with $\frac{1}{p}>1$ and its conjugate Hölder exponent $\frac{1}{1-p}$,
giving %
{} %
\begin{align*}
\left\Vert \left(\frac{\left|Y\right|Z}{Z}\right)^{p}\right\Vert _{1} & \leq\left\Vert \left(\left|Y\right|Z\right)^{^{p}}\right\Vert _{\frac{1}{p}}\cdot\left\Vert \frac{1}{Z^{p}}\right\Vert _{\frac{1}{1-p}}=\bigl\Vert\left|Y\right|Z\bigr\Vert_{1}^{p}\cdot\left\Vert Z\right\Vert _{q}^{-p}=1
\end{align*}
and thus the statement.

We now derive the reverse Minkowski inequality by employing the reverse
Hölder inequality. Let $Z_{1},\,Z_{2}\in L^{\infty}$ be positive,
then 
\begin{align*}
\left\Vert Z_{1}+Z_{2}\right\Vert _{q}^{q} & =\E Z_{1}\left(Z_{1}+Z_{2}\right)^{q-1}+\E Z_{2}\left(Z_{1}+Z_{2}\right)^{q-1}.
\end{align*}
An application of the reverse Hölder inequality with conjugate exponents
$q$ and $\frac{q}{q-1}$ gives 
\begin{align*}
\E Z_{1}\left(Z_{1}+Z_{2}\right)^{q-1}+\E Z_{2}\left(Z_{1}+Z_{2}\right)^{q-1} & \geq\left(\left\Vert Z_{1}\right\Vert _{q}+\left\Vert Z_{2}\right\Vert _{q}\right)\left\Vert \left(Z_{1}+Z_{2}\right)^{q-1}\right\Vert _{\frac{q}{q-1}}\\
 & =\left(\left\Vert Z_{1}\right\Vert _{q}+\left\Vert Z_{2}\right\Vert _{q}\right)\left(\E\left(Z_{1}+Z_{2}\right)^{q}\right)^{\frac{q-1}{q}}
\end{align*}
from which the inequality $\left\Vert Z_{1}+Z_{2}\right\Vert _{q}\geq\left\Vert Z_{1}\right\Vert _{q}+\left\Vert Z_{2}\right\Vert _{q}$
follows. 
\end{proof}
\begin{rem}
The functional $\left\Vert \cdot\right\Vert _{p^{\prime}}$ is not
convex for $p^{\prime}<1,$ hence one might assume that the set 
\begin{equation}
\left\{ Z\in L^{1}\colon Z\geq0,\ \E Z=1,\ H_{p^{\prime}}(Z)\leq\log\frac{1}{1-\alpha}\right\} \label{eq:5-1}
\end{equation}
of feasible densities in~\eqref{eq: EVaRDef1} is not convex for
$p^{\prime}<1$ and $\alpha>0$. However, $H_{p^{\prime}}(Z)\leq\log\frac{1}{1-\alpha}$
rewrites as $\left\Vert Z\right\Vert _{p^{\prime}}\geq\left(1-\alpha\right)^{-\frac{1}{p^{\prime}}}$
for $p^{\prime}<1$. Thus the reverse Minkowski inequality guarantees
that the set of feasible densities~\eqref{eq:5-1} is convex even
for $p^{\prime}<1$.
\end{rem}

\begin{thm}
\label{thm:RVaR(p<1)}The domain of the risk measure $\EVaR_{\alpha}^{p}(\cdot)$
for $p<1$ is $L^{\infty}$, i.e., $\EVaR_{\alpha}^{p}\left(\left|Y\right|\right)<\infty$
if and only if $Y$ is bounded. Furthermore, the entropic risk measure
$\EVaR_{\alpha}^{p}$ collapses to the essential supremum for $0<p<1$.
\end{thm}

\begin{proof}
For $p\in(0,1)$ the conjugate Hölder exponent $p^{\prime}$ is negative
and we may thus assume that $Z>0$. By Remark~\ref{rem:Pos-Entr}
we conclude that $H_{p^{\prime}}(Z)\leq0$. The constraint $H_{p^{\prime}}(Z)\le\log\frac{1}{1-\alpha}$
thus is trivial, as $\log\frac{1}{1-\alpha}\ge0$ and it follows that
the entropic risk measure $\EVaR_{\alpha}^{p}$ reduces to
\[
\EVaR_{\alpha}^{p}(Y)=\sup\left\{ \E YZ\colon Z\geq0,\,\E Z=1\right\} =\esssup Y
\]
in this case. 

Let us now consider the case $p<0$. Then its Hölder conjugate exponent
satisfies $0<p^{\prime}<1$ and the constraint $H_{p^{\prime}}(Z)\le\log\frac{1}{1-\alpha}$
is equivalent to $\E Z^{p^{\prime}}\ge\left(\frac{1}{1-\alpha}\right)^{p^{\prime}-1}=\left(1-\alpha\right)^{1-p^{\prime}}$.
We may choose $\kappa>1$ large enough so that
\[
\left(1-\frac{1}{\kappa}\right)^{p^{\prime}}>\left(1-\alpha\right)^{1-p^{\prime}}.
\]
For $\beta\in(0,1)$ consider random variables $Z_{\beta}$ with 
\[
P\left(Z_{\beta}=\frac{1}{\kappa\beta}\right)=\beta\text{ and }P\left(Z_{\beta}=\frac{1-\frac{1}{\kappa}}{1-\beta}\right)=1-\beta.
\]
The random variable $Z_{\beta}$ is a density, as $Z_{\beta}>0$ and
$\E Z_{\beta}=1$. For the random variables $Z_{\beta}$ it thus holds
that 
\[
\E Z_{\beta}^{p^{\prime}}=\beta\left(\frac{1}{\kappa\beta}\right)^{p^{\prime}}+(1-\mathbf{\beta})\left(\frac{1-\frac{1}{\kappa}}{1-\beta}\right)^{p^{\prime}}\xrightarrow[\beta\to0]{}\left(1-\frac{1}{\kappa}\right)^{p^{\prime}}>\left(1-\alpha\right)^{1-p^{\prime}}.
\]
We thus may choose $\hat{\beta}<1$ so that for $\beta<\hat{\beta}$
we have that $\mathbb{E}Z_{\beta}^{p^{\prime}}>\left(1-\alpha\right)^{1-p^{\prime}}$.

Finally let $Y$ be an unbounded random variable. Without loss of
generality we may assume that $Y\ge0$. Then, for each $n\in\mathbb{N}$,
the set $B_{n}:=\left\{ Y\ge n\right\} $ has strictly positive probability
and we set $\beta_{n}:=P(B_{n})>0$. The variable
\[
Z_{\beta_{n}}(\omega):=\begin{cases}
\frac{1}{\kappa\beta_{n}} & \text{if }\omega\in B_{n}\\
\frac{1-\frac{1}{\kappa}}{1-\beta_{n}} & \text{if }\omega\not\in B_{n}
\end{cases}
\]
is feasible and it holds that $\E YZ_{\beta_{n}}\ge\frac{n}{\kappa}$
and thus $\EVaR_{\alpha}^{p}(Y)\ge\frac{n}{\kappa}$. This proves
that $Y\notin L^{\infty}$ implies $\EVaR_{\alpha}^{p}(Y)=\infty$.
The converse implication follows directly from Hölder's inequality. 
\end{proof}
It is now clear that for $p>1$ the risk measures $\EVaR_{\alpha}^{p}$
have the domain $L^{p}$ and for these spaces the dual spaces are
known, they are $L^{p^{\prime}}$ spaces, where $p^{\prime}$ is the
Hölder conjugate of $p$. For $p=\infty$ we have the special case
that $\EVaR$ can be defined on a space larger than $L^{\infty}$,
this is studied in \citet{AhmadiPi-R}. 

In what follows we address the duality relations of the Entropic Value-at-Risk
for $p<1$. 

\section{\label{sec:Duality}Dual representation of entropic risk measures}

In this section we develop representations of entropic risk measures
which are dual to the expression given in Definition~\ref{def:Entropic}.
This characterization allows us to deduce continuity properties of
$\EVaR$ as well as to compare the entropic risk measures with Hölder
norms. In what follows we discuss the three cases $p\ge1$, $p\in(0,1)$
and $p<0$ separately.

\subsection{Infimum representation for $p\geq1$}

The following theorem is originally due to~\citet{Dentcheva2010-R}.
We state the result as it is similar to duality representations given
below. We further use it to construct an explicit characterization
of the dual norm and its corresponding Hahn\textendash Banach functionals
in Section~\ref{sec:Dual-Norms} below. 
\begin{thm}[Infimum representation for $p\ge1$, cf.\ \citet{Dentcheva2010-R}]
\label{thm:primal1}Let $\alpha\in\left(0,1\right)$, then the Entropic
Value-at-Risk based on Rényi entropy for $p\in[1,\infty)$ has the
representation
\begin{equation}
\EVaR_{\alpha}^{p}(Y)=\inf_{t\in\mathbb{R}}\left\{ t+\left(\frac{1}{1-\alpha}\right)^{\nicefrac{1}{p}}\cdot\left\Vert \left(Y-t\right)_{+}\right\Vert _{p}\right\} .\label{eq: EVaRinf1}
\end{equation}
\end{thm}

\begin{rem}
\label{rem:7}Note that the previous setting includes the case $p=1$
as a special case and we have the identity 
\[
\AVaR_{\alpha}(Y)=\inf_{t\in\mathbb{R}}\left\{ t+\frac{1}{1-\alpha}\cdot\E(Y-t)_{+}\right\} 
\]
given in~\citet{RuszOgryczak}. The result~\eqref{eq: EVaRinf1}
deduces as well from 
\begin{equation}
\EVaR_{\alpha}^{p}(Y)=\sup\left\{ \E YZ\colon Z\ge0,\,\E Z=1\text{ and }\left\Vert Z\right\Vert _{p^{\prime}}\le\left(\frac{1}{1-\alpha}\right)^{\frac{p^{\prime}-1}{p^{\prime}}}\right\} \label{eq:3-2}
\end{equation}
and \citet[Theorem 3.1]{Pichler2017-R}.

\citet{Dentcheva2010-R} also relate the optimal density $Z^{*}$
of~\eqref{eq: EVaRDef1} and the optimizer $t^{*}$ of~\eqref{eq: EVaRinf1}
by 
\begin{equation}
Z^{*}=\frac{\left(Y-t^{*}\right)_{+}^{p-1}}{\E\left(Y-t^{*}\right)_{+}^{p-1}},\label{eq:7-4}
\end{equation}
if $t^{*}<\esssup Y$. For $t^{*}=\esssup Y$ they give the optimal
density $Z^{*}=\frac{1}{p_{\text{max}}}\one_{\left\{ Y=\esssup Y\right\} }$
with $p_{\text{max}}:=P\left(Y=\esssup Y\right)$. 
\end{rem}

\subsection{Infimum representation for $p<0$}

To elaborate the dual representation in analogy to~\eqref{eq: EVaRinf1}
 for $p<0$ we discuss the function $f\colon(\esssup Y,\infty)\to\mathbb{R}$
given by
\begin{equation}
f(t)=t-\left(\frac{1}{1-\alpha}\right)^{\nicefrac{1}{p}}\cdot\left\Vert t-Y\right\Vert _{p},\label{eq: ObjectiveRVaRinf}
\end{equation}
where $p<0$ and $Y\in L^{\infty}.$ The function $f(\cdot)$ has
the following property.
\begin{prop}
\label{prop:objective}The following are equivalent for $p<0$ and
$Y\in L^{\infty}$:
\begin{enumerate}
\item \label{enu:2-1}The infimum in~\eqref{eq: ObjectiveRVaRinf} is attained
at some $t^{*}>\esssup Y$,
\item \label{enu:2}$P\left(Y=\esssup Y\right)<1-\alpha$. 
\end{enumerate}
Furthermore,\textup{ }if $t^{*}=\esssup Y$, then  it holds that \textup{$\inf_{t>\esssup Y}\,f(t)=\esssup Y$}. 

\end{prop}

\begin{proof}
For $p<0$, the function $\left\Vert \cdot\right\Vert _{p}$  is concave
by the reverse Minkowski inequality. The function~\eqref{eq: ObjectiveRVaRinf}
thus is convex and hence almost everywhere differentiable. For $t\in\mathbb{R}$
large enough the objective is monotone increasing, as we have
\begin{align*}
f\left(t\right) & =t-\left(\frac{1}{1-\alpha}\right)^{\frac{1}{p}}\left(\E\left(t-Y\right){}^{p}\right)^{\frac{1}{p}}=t-\left(\frac{1}{1-\alpha}\right)^{\frac{1}{p}}\cdot t\cdot\left(\E1-\frac{pY}{t}\right)^{\frac{1}{p}}+\mathcal{O}\left(\nicefrac{1}{t}\right)\\
 & =t-\left(\frac{1}{1-\alpha}\right)^{\frac{1}{p}}\left(t-\E Y\right)+\mathcal{O}\left(\nicefrac{1}{t}\right),
\end{align*}
by successive Taylor series expansions; the infimum is hence attained
at some $t^{*}<\infty$, as $\left(\frac{1}{1-\alpha}\right)^{\frac{1}{p}}<1$. 

The derivative $f^{\prime}(t)=1-(1-\alpha)^{-1/p}\left(\E\left(t-Y\right){}^{p}\right)^{\frac{1}{p}-1}\cdot\E\left(t-Y\right){}^{p-1}$
is negative if and only if $\left(1-\alpha\right){}^{1/p}<\left(\E\left(t-Y\right){}^{p}\right)^{\frac{1-p}{p}}\cdot\E\left(t-Y\right){}^{p-1}.$
Now note that $1-p>0$ so that we get $\left(1-\alpha\right){}^{\frac{1}{p(1-p)}}\cdot\left\Vert t-Y\right\Vert _{p-1}<\left\Vert t-Y\right\Vert _{p}$.
Define $A_{\varepsilon}:=\left\{ Y+\varepsilon\ge\esssup Y\right\} $
and set $Y_{\text{max}}:=\esssup Y$. Then we have for large $t$
that 
\[
\left(t-Y_{\text{max}}\right){}^{p-1}P\left(A_{\varepsilon}\right)+\varepsilon^{p-1}P\left(A_{\varepsilon}^{\mathsf{c}}\right)\ge\E\left(t-Y\right)^{p-1}
\]
whenever $t>Y_{\text{max}}$. This implies
\begin{align*}
\left(1-\alpha\right){}^{\frac{1}{p(1-p)}}\cdot & \left(\left(t-Y_{\text{max}}\right)^{p-1}P\left(A_{\varepsilon}\right)+\varepsilon^{p-1}P\left(A_{\varepsilon}^{\mathsf{c}}\right)\right)^{\frac{1}{p-1}}\\
\le & \left(1-\alpha\right){}^{\frac{1}{p(1-p)}}\cdot\left\Vert t-Y\right\Vert _{p-1}<\left\Vert t-Y\right\Vert _{p}\le\left(\left(t-Y_{\text{max}}\right)^{p}P(A_{\varepsilon})\right)^{\frac{1}{p}}.
\end{align*}
Dividing the quantity by $t-Y_{\text{max}}$ we obtain further
\[
\left(1-\alpha\right){}^{\frac{1}{p(1-p)}}\cdot\left(P\left(A_{\varepsilon}\right)+\varepsilon^{p-1}P\left(A_{\varepsilon}^{\mathsf{c}}\right)\left(t-Y_{\text{max}}\right)^{1-p}\right)^{\frac{1}{p-1}}<P\left(A_{\varepsilon}\right)^{\frac{1}{p}}.
\]
Letting $t\searrow Y_{\text{max}}$ yields $\left(1-\alpha\right){}^{\frac{1}{p(1-p)}}\cdot P\left(A_{\varepsilon}\right)^{\frac{1}{p-1}}<P\left(A_{\varepsilon}\right)^{\frac{1}{p}}$,
which is equivalent to $\left(1-\alpha\right)>P\left(A_{\varepsilon}\right)$
and letting $\varepsilon\searrow0$ give $1-\alpha>P\left(A_{0}\right)=P\left(Y=\esssup Y\right)$.
Note that strict inequality holds, since $P\left(A_{\varepsilon}\right)\geq P\left(A_{0}\right)$
for all $\varepsilon>0$. Therefore $t^{*}>\esssup Y$ implies that
$1-\alpha>P\left(Y=\esssup Y\right)$. 

The converse implication~$\ref{enu:2}\implies\ref{enu:2-1}$ is proven
similarly. 

To see the remaining statement set $A:=\left\{ Y=\esssup Y\right\} $.
By the previous results we know that the objective function $f$ is
increasing on its domain. Therefore the infimum of~\eqref{eq: ObjectiveRVaRinf}
is a limit and
\begin{align*}
\inf_{t>\esssup Y}\left\{ t-\left(\frac{1}{1-\alpha}\right)^{\nicefrac{1}{p}}\cdot\left\Vert t-Y\right\Vert _{p}\right\}  & =\lim_{t\searrow\esssup Y}\:t-\lim_{t\searrow\esssup Y}\:\left(\frac{1}{1-\alpha}\right)^{\nicefrac{1}{p}}\cdot\left\Vert t-Y\right\Vert _{p}.
\end{align*}
But for the second limit we have
\begin{align*}
0\le\lim_{t\searrow\esssup Y}\left\Vert t-Y\right\Vert _{p} & \leq\lim_{t\searrow\esssup Y}\left(\left(t-\esssup Y\right)^{p}P\left(A\right)\right)^{\frac{1}{p}}\\
 & =\lim_{t\searrow\esssup Y}\left(t-\esssup Y\right)P\left(A\right)^{\frac{1}{p}}=0,
\end{align*}
which proves that $\inf_{t>\esssup Y}\,f(t)=\esssup Y$ and concludes
the proof. 
\end{proof}
Using the characterization of the optimal value of $f(\cdot)$ in~\eqref{eq: ObjectiveRVaRinf}
we are now ready to prove the infimum representation of $\EVaR_{\alpha}^{p}$
for $p<0$.
\begin{thm}[Infimum representation for $p<0$]
\label{thm:primal2}Let $\alpha\in(0,1)$, then the Entropic Value-at-Risk
based on Rényi entropy $\left(\EVaR_{\alpha}^{p}\right)$ for $p<0$
has the representation 
\begin{equation}
\EVaR_{\alpha}^{p}(Y)=\inf_{t>\esssup Y}\left\{ t-\left(\frac{1}{1-\alpha}\right)^{\nicefrac{1}{p}}\cdot\left\Vert t-Y\right\Vert _{p}\right\} .\label{eq:EVaRinf}
\end{equation}
\end{thm}

\begin{proof}
Let $f\left(\cdot\right)$ denote the objective function in~\eqref{eq:EVaRinf}.
From the previous proposition it is clear that the minimizer $t^{*}$
of $f(\cdot)$ either satisfies $t^{*}=\esssup Y$ or $t^{*}>\esssup Y$.
First assume that $t^{*}>\esssup Y$. The random variable $Z:=\frac{1}{c}\left(t^{*}-Y\right){}^{p-1}$
with $c:=\E\left(t^{*}-Y\right)^{p-1}$ has expectation $1$ and $Z>0$,
i.e., $Z$ is a density. By definition of $Z$, Hölder's inequality
is an equality for $Z$ and $(t^{*}-Y)$, i.e., $\E\left(t^{*}-Y\right)Z=\left\Vert t^{*}-Y\right\Vert _{p}\left\Vert Z\right\Vert _{p^{\prime}}$,
as 
\[
\left|Z\right|^{p^{\prime}}\E\left(t^{*}-Y\right){}^{p}=\left|t^{*}-Y\right|^{p}\,\E Z^{p^{\prime}}.
\]
Furthermore, for the optimizer $t^{*}$ of the objective function
$f(\cdot)$ it holds that $f^{\prime}(t^{*})=0$, which is  
\[
\left(1-\alpha\right){}^{\frac{1}{p}}=\E\left(t^{*}-Y\right)^{p-1}\left(\E\left(t^{*}-Y\right)^{p}\right)^{-\frac{1}{p'}}.
\]
Therefore,
\begin{align*}
\E YZ & =t^{*}-\E\left(t^{*}-Y\right)Z=t^{*}-\left\Vert t^{*}-Y\right\Vert _{p}\left\Vert Z\right\Vert _{p^{\prime}}=t^{*}-\left(\frac{1}{1-\alpha}\right)^{\frac{1}{p}}\left\Vert t^{*}-Y\right\Vert _{p}\\
 & =\inf_{t>\esssup Y}\left\{ t-\left(\frac{1}{1-\alpha}\right)^{\nicefrac{1}{p}}\cdot\left\Vert t-Y\right\Vert _{p}\right\} ,
\end{align*}
establishing that $\EVaR_{\alpha}^{p}(Y)\geq\inf_{t>\esssup Y}\left\{ t-\left(\frac{1}{1-\alpha}\right)^{\nicefrac{1}{p}}\cdot\left\Vert t-Y\right\Vert _{p}\right\} $.
Suppose for a moment that $Z$ is not a maximizing density in~\eqref{eq: EVaRDef1},
then there is another density $\widehat{Z}$ satisfying the moment
constraints which maximizes~\eqref{eq: EVaRDef1}, but
\begin{align*}
\E Y\widehat{Z} & =t^{*}-\E\left(t^{*}-Y\right)\widehat{Z}<t^{*}-\lVert t^{*}-Y\rVert_{p}\lVert\widehat{Z}\rVert_{p^{\prime}}.
\end{align*}
Since $H_{p^{\prime}}(Z)\leq\log\left(\frac{1}{1-\alpha}\right)$
is equivalent to $\left\Vert Z\right\Vert _{p^{\prime}}\geq\left(\frac{1}{1-\alpha}\right)^{\frac{1}{p}}$,
it follows that 
\begin{equation}
t^{*}-\lVert t^{*}-Y\rVert_{p}\lVert\widehat{Z}\rVert_{p^{\prime}}<t^{*}-\left(\frac{1}{1-\alpha}\right)^{\frac{1}{p}}\left\Vert t^{*}-Y\right\Vert _{p}=\E YZ,\label{eq: weak duality}
\end{equation}
contradicting the assumption that $\widehat{Z}$ maximizes~\eqref{eq: EVaRDef1}
and thus $\EVaR_{\alpha}^{p}(Y)=\E YZ$.

Now assume that $t^{*}=\esssup Y$. By the previous proposition it
follows that $P(Y=\esssup Y)\geq1-\alpha$ and therefore define 
\[
Z:=P\left(Y=\esssup Y\right)^{-1}\one_{\left\{ Y=\esssup Y\right\} }.
\]
Then $\E Z=1$, $\left\Vert Z\right\Vert _{p^{\prime}}=P\left(Y=\esssup Y\right)^{\frac{1-p^{\prime}}{p^{\prime}}}>\left(1-\alpha\right)^{\frac{1-p^{\prime}}{p^{\prime}}}$
and $\E YZ=\esssup Y$. By the previous proposition we have $\EVaR_{\alpha}^{p}(Y)\geq\inf_{x>\esssup Y}\left\{ t-\left(\frac{1}{1-\alpha}\right)^{\nicefrac{1}{p}}\cdot\left\Vert t-Y\right\Vert _{p}\right\} =\esssup Y$.
Now consider any density $\widehat{Z}$ satisfying the moment constraints
and let $t^{*}=\esssup Y$. Then 
\begin{align*}
\E Y\widehat{Z} & =t^{*}-\E\left(t^{*}-Y\right)\widehat{Z}\leq t^{*}=\esssup Y=\inf_{x>\esssup Y}\left\{ t-\left(\frac{1}{1-\alpha}\right)^{\nicefrac{1}{p}}\cdot\left\Vert t-Y\right\Vert _{p}\right\} ,
\end{align*}
where the last equation follows by the assumption $t^{*}=\esssup Y$
and Proposition~\ref{prop:objective}. This establishes the infimum
representation.
\end{proof}

\subsection{Infimum representation for $p\in\left(0,1\right)$}

From Theorem~\ref{thm:RVaR(p<1)} we know that for $0<p<1$ the entropic
risk measure $\EVaR_{\alpha}^{p}$ does not depend on $\alpha$ or
$p$. The corresponding infimum representation is 
\[
\EVaR_{\alpha}^{p}(Y)=\left\Vert Y\right\Vert _{\infty}=\inf_{t\geq\esssup Y}\left\{ t+\left\Vert \left(Y-t\right)_{+}\right\Vert _{p}\right\} ,
\]
as the infimum is always attained for $t=\esssup Y$. As this case
is trivial we will not consider it throughout the remainder of this
paper. 

\section{\label{sec:Analytic-prop}Monotonicity and EV@R-norms }

Based on the properties of the Rényi entropy and the infimum representation
of $\EVaR$ we deduce limiting risk measures of $\EVaR^{p}$ as well
as a general monotonicity in the dual parameter $p^{\prime}$. We
further show that the Entropic Value-at-Risk based on Rényi entropy
is convex in its dual order $p^{\prime}$. 
\begin{lem}[The special case $p=1$]
\label{lem:Extreme1}For $\alpha\in\left(0,1\right)$ and $Y\in L^{1}$
we have
\[
\EVaR_{\alpha}^{p}(Y)\xrightarrow[p\downarrow1]{}\AVaR_{\alpha}(Y).
\]
\end{lem}

\begin{proof}
From the definition of Rényi entropy we have
\[
\lim_{p^{\prime}\to\infty}H_{p^{\prime}}(Z)=\lim_{p^{\prime}\to\infty}\frac{p^{\prime}}{p^{\prime}-1}\log\left\Vert Z\right\Vert _{p^{\prime}}=\log\left\Vert Z\right\Vert _{\infty},
\]
so that the inequality $\log\left\Vert Z\right\Vert _{\infty}=\lim_{p^{\prime}\to\infty}H_{p^{\prime}}(Z)\le\log\frac{1}{1-\alpha}$
is satisfied for every $Z$ satisfying the constraint in~\eqref{eq: EVaRDef1}
and $p>1$. Therefore $Z\le\frac{1}{1-\alpha}$, and consequently
\[
\lim_{p\to1}\EVaR_{\alpha}^{p}(Y)\le\AVaR_{\alpha}(Y).
\]
Consider further the generalized indicator function
\[
\one_{\left\{ Y\geq\VaR_{\alpha}(Y)\right\} }^{\alpha}:=\begin{cases}
\one_{\left\{ Y\geq\VaR_{\alpha}(Y)\right\} } & \text{if }P(Y=\VaR_{\alpha}(Y))=0;\\
\one_{\left\{ Y>\VaR_{\alpha}(Y)\right\} }+\frac{P(Y\leq\VaR_{\alpha}(Y))-\alpha}{P(Y=\VaR_{\alpha}(Y))}\one_{\left\{ Y=\VaR_{\alpha}(Y)\right\} } & \text{if }P(Y=\VaR_{\alpha}(Y))>0
\end{cases}
\]
and define the random variable $Z=\frac{1}{1-\alpha}\one_{\left\{ Y\geq\VaR_{\alpha}(Y)\right\} }^{\alpha}$.
Then $Z$ is a density with $H_{\infty}(Z)=\log\frac{1}{1-\alpha}$
and thus $\lim_{p\to1}\EVaR_{\alpha}^{p}(Y)\ge\AVaR_{\alpha}(Y)$.
Hence the assertion.
\end{proof}
\begin{lem}[Special case $p=0$]
\label{lem:Extreme2}For $\alpha\in\left(0,1\right)$ and $Y\in L^{\infty}$
it holds that 
\[
\EVaR_{\alpha}^{p}(Y)\xrightarrow[p\uparrow0]{}\left\Vert Y\right\Vert _{\infty}.
\]
\end{lem}

\begin{proof}
Set $A_{\varepsilon}:=\left\{ Y>\left\Vert Y\right\Vert _{\infty}-\varepsilon\right\} $
for $\varepsilon>0$. Set $\lambda:=\frac{1}{2}P\left(A_{\varepsilon}\right)$
and consider the density 
\[
Z:=\frac{1}{2\lambda}\one_{A_{\varepsilon}}.
\]
Note now that $x^{p^{\prime}}\xrightarrow[p^{\prime}\to0]{}1$ for
every $x>0$, so it is possible to find $p^{\prime}>0$ small enough
so that $\frac{1}{p^{\prime}-1}\log\E Z^{p^{\prime}}\le\log\frac{1}{1-\alpha}$.
It holds that $\E\left|Y\right|Z\ge\frac{1}{2\lambda}\left(\left\Vert Y\right\Vert _{\infty}-\varepsilon\right)P\left(A_{\varepsilon}\right)=\left\Vert Y\right\Vert _{\infty}-\varepsilon$,
from which the assertion follows. 
\end{proof}

\subsection{Monotonicity and Convexity}

The following theorem states that the Average Value-at-Risk and the
Entropic Value-at-Risk are extremal cases for the risk measure based
on Rényi entropy. It is more convenient to state the result in $p^{\prime}$
than $p.$ This is due to monotonicity of $p^{\prime}\mapsto H_{p^{\prime}}(Z)$
from Lemma~\ref{lem:R=0000E9nyi-Entropy-Mono} and the fact that
$p^{\prime}$ is less than $0$ whenever $p\in(0,1)$.
\begin{thm}[Monotonicity]
\label{thm:Monotonicity}Let $\alpha\in\left(0,1\right)$, then for
$p_{1}^{\prime}<0<p_{2}^{\prime}<1\leq p_{3}^{\prime}\leq p_{4}^{\prime}$
and their corresponding Hölder conjugates $p_{i}=\frac{p_{i}^{\prime}}{p_{i}^{\prime}-1}$
it holds that
\[
\AVaR_{\alpha}(\cdot)\leq\EVaR_{\alpha}^{p_{4}}(\cdot)\leq\EVaR_{\alpha}^{p_{3}}(\cdot)\leq\EVaR_{\alpha}(\cdot)\leq\EVaR_{\alpha}^{p_{2}}(\cdot)\leq\EVaR_{\alpha}^{p_{1}}(\cdot)=\left\Vert \cdot\right\Vert _{\infty}.
\]
\[
\left\Vert \cdot\right\Vert _{\infty}\geq\EVaR_{\alpha}^{p_{1}}(\cdot)\geq\EVaR_{\alpha}^{p_{2}}(\cdot)\geq\EVaR_{\alpha}(\cdot)\geq\EVaR_{\alpha}^{p_{3}}(\cdot)\geq\EVaR_{\alpha}^{p_{4}}(\cdot)\geq\AVaR_{\alpha}(\cdot)
\]
The mapping $p^{\prime}\mapsto\EVaR_{\alpha}^{p}(Y)$ is continuous
for all $Y$, for which the expression is finite.
\end{thm}

\begin{proof}
The result follows from Lemmas~\ref{lem:Extreme1},~\ref{lem:Extreme2}
and~\ref{lem:R=0000E9nyi-Entropy-Mono}. The last equality follows
from Theorem~\ref{thm:RVaR(p<1)}. Continuity follows from Lemma~\ref{prop: Hqconvex}.
\end{proof}
From Theorem~\ref{thm:Monotonicity} we know that $p^{\prime}\mapsto\EVaR_{\alpha}^{p}$
is decreasing. We now show that this mapping is not only monotone
but logarithmically convex.
\begin{thm}[Log-convexity of the Entropic Value-at-Risk]
\label{thm:log-convex}For $1<p_{0},\,p_{1}$ and $0\leq\lambda\leq1$
define $p_{\lambda}^{\prime}:=(1-\lambda)p_{0}^{\prime}+\lambda p_{1}^{\prime}$,
where $p_{0}^{\prime}$ ($p_{1}^{\prime}$, resp.) is the Hölder conjugate
exponent of $p_{0}$ ($p_{1}$, resp.). Then, for $Y\in L^{\infty}$,
$\alpha\in(0,1)$ and $p_{\lambda}:=\frac{p_{\lambda}^{\prime}}{p_{\lambda}^{\prime}-1}$
we have that
\[
\EVaR_{\alpha}^{p_{\lambda}}\left(\left|Y\right|\right)\leq\left(\EVaR_{\alpha}^{p_{0}}\left(\left|Y\right|\right)\right)^{1-\lambda}\cdot\left(\EVaR_{\alpha}^{p_{1}}\left(\left|Y\right|\right)\right)^{\lambda}.
\]
That is, the Entropic Value-at-Risk is logarithmically convex in its
conjugate order.
\end{thm}

\begin{proof}
The rather technical proof can be found in the appendix. There we
first derive a compact representation of the derivative $\frac{\mathrm{d}}{\mathrm{d}p^{\prime}}\EVaR_{\alpha}^{\frac{p^{\prime}}{p^{\prime}-1}}(Y)$
and consequently show first convexity and then log-convexity. 
\end{proof}
\begin{rem}
We emphasize that the above result on logarithmic convexity of the
function $p^{\prime}\mapsto\EVaR_{\alpha}^{\frac{p^{\prime}}{^{p^{\prime}-1}}}(Y)$
does not extent to the case $p^{\prime}\in(0,1)$. For $p^{\prime}<0$,
in contrast, we have $\EVaR_{\alpha}^{\frac{p^{\prime}}{^{p^{\prime}-1}}}(Y)=\esssup(Y)$
and convexity hence is obvious. 
\end{rem}

\subsection{Comparison with Hölder norms}

The remainder of this section is concerned with the norms generated
by $\EVaR$. We show that the $\EVaR_{\alpha}^{p}$-norm is equivalent
to Hölder norms ($p\not=\infty$), irrespective of the confidence
level $\alpha.$ The norms are equivalent for varying confidence level
$\alpha$ and $p$ fixed. 
\begin{thm}[Comparison with $L^{p}$]
\label{thm:Normequi1}The space $\left(L^{p},\EVaR_{\alpha}^{p}\left(\left|\cdot\right|\right)\right)$
is a Banach space for each $\alpha\in\left(0,1\right)$ and $p\in(1,\infty)$.
Furthermore, the norms $\left\Vert \cdot\right\Vert _{p}$ and $\EVaR_{\alpha}^{p}\left(\left|\cdot\right|\right)$
are equivalent, i.e., it holds that 
\begin{equation}
C\left\Vert \cdot\right\Vert _{p}\leq\EVaR_{\alpha}^{p}\left(\left|\cdot\right|\right)\leq\left(\frac{1}{1-\alpha}\right)^{\frac{1}{p}}\left\Vert \cdot\right\Vert _{p},\label{eq: Normequi1}
\end{equation}
where $C=1\wedge\left(\left(\frac{1}{1-\alpha}\right)^{\frac{1}{p-1}}-1\right)^{\frac{p-1}{p}}$.
The inequality~\eqref{eq: Normequi1} is sharp. 
\end{thm}

\begin{proof}
The right inequality follows directly from Hölder's inequality, cf.~\eqref{eq:7-1}.
To see that the right inequality is sharp consider the random variable
$Y=\one_{A}$ with $P\left(A\right)=1-\alpha$. Then it holds that
$\EVaR_{\alpha}^{p}(Y)=1$ and $\left\Vert Y\right\Vert _{p}=\left(1-\alpha\right)^{\frac{1}{p}}$. 

For the left inequality consider first the case where the optimal
$t^{*}\geq0$ so that 
\[
\EVaR_{\alpha}^{p}(Y)\geq\inf_{t>0}t+\left\Vert \left(Y-t\right)_{+}\right\Vert _{p}=\left\Vert Y\right\Vert _{p}.
\]
The case $t^{*}<0$ is more complicated. Consider a simple random
variable $Y=\sum_{i}\lambda_{i}\one_{A_{i}}$, then 
\[
C\left\Vert Y\right\Vert _{p}=C\left\Vert \sum_{i}\lambda_{i}\one_{A_{i}}\right\Vert _{p}\le C\sum_{i}\lambda_{i}\left\Vert \one_{A_{i}}\right\Vert _{p}\le\sum_{i}\lambda_{i}\EVaR_{\alpha}^{p}(\one_{A_{i}})=\EVaR(Y),
\]
therefore it is enough to consider random variables of the form $Y=C\one_{A}$.
Without loss of generality we may assume that $\left\Vert Y\right\Vert _{p}=1$
and consider $Y_{A}=\left(\frac{1}{\varepsilon}\right)^{\frac{1}{p}}\one_{A}$
where $\varepsilon:=P(A)$. Define the quantity $t^{*}=-\left(\left(1-\alpha\right)^{\frac{1}{1-p}}-1\right)^{-\frac{1}{p}}$,
which is negative for $\alpha<1-2^{1-p}$. Hence, for $\alpha<1-2^{1-p}$,
we have that 
\[
\EVaR_{\alpha}^{p}\left(Y_{\varepsilon}\right)\leq t^{*}+\left(\frac{1}{1-\alpha}\right)^{\frac{1}{p}}\left(\left(1-\varepsilon\right)\left(-t^{*}\right)^{p}+\varepsilon\left(\left(\frac{1}{\varepsilon}\right)^{\frac{1}{p}}-t^{*}\right)^{p}\right)^{\frac{1}{p}}.
\]
The right hand side is differentiable in $\varepsilon$ with derivative
\begin{align*}
\frac{\mathrm{d}}{\mathrm{d}\varepsilon}f(\varepsilon)= & \left(\frac{1}{1-\alpha}\right)^{\frac{1}{p}}\frac{1}{p}\left[\left(1-\varepsilon\right)\left(-t^{*}\right)^{p}+\varepsilon\left(\left(\varepsilon\right)^{-\frac{1}{p}}-t^{*}\right)\right]^{\frac{1}{p}-1}\\
 & \times\left[-\left(-t^{*}\right)^{p}+\left(\left(\frac{1}{\varepsilon}\right)^{\frac{1}{p}}-t^{*}\right)^{p}+\varepsilon p\left(-\frac{1}{p}\varepsilon^{-\frac{1}{p}-1}\right)\left(\left(\frac{1}{\varepsilon}\right)^{\frac{1}{p}}-t^{*}\right)^{p-1}\right].
\end{align*}
The first factor is positive, so that the derivative $\frac{\mathrm{d}}{\mathrm{d}\varepsilon}f(\varepsilon)$
is positive, if and only if
\begin{align}
0 & \leq-\left(-t^{*}\right)^{p}+\left(\left(\frac{1}{\varepsilon}\right)^{\frac{1}{p}}-t^{*}\right)^{p}+\varepsilon p\left(-\frac{1}{p}\varepsilon^{-\frac{1}{p}-1}\right)\left(\left(\frac{1}{\varepsilon}\right)^{\frac{1}{p}}-t^{*}\right)^{p-1}\label{eq:7-2}\\
 & =\left(\left(\frac{1}{\varepsilon}\right)^{\frac{1}{p}}-t^{*}\right)^{p-1}\left(-\varepsilon^{-\frac{1}{p}}+\varepsilon^{-\frac{1}{p}}-t^{*}\right)-\left(-t^{*}\right)^{p}\nonumber 
\end{align}
However, this is the case, as it is a consequence of the binomial
theorem. We now consider the limit 
\[
\lim_{\varepsilon\to0}\:t^{*}+\left(\frac{1}{1-\alpha}\right)^{\frac{1}{p}}\left(\left(1-\varepsilon\right)\left(-t^{*}\right)^{p}+\varepsilon\left(\left(\frac{1}{\varepsilon}\right)^{\frac{1}{p}}-t^{*}\right)^{p}\right)^{\frac{1}{p}}=t^{*}+\left(\frac{1}{1-\alpha}\right)^{\frac{1}{p}}\left(\left(-t^{*}\right)^{p}+1\right)^{\frac{1}{p}},
\]
which is a lower bound for $\EVaR_{\alpha}^{p}\left(Y_{\varepsilon}\right)$.
This is the optimal bound, since 
\[
t^{*}\in\argmin_{t\in\mathbb{R}}\:t+\left(\frac{1}{1-\alpha}\right)^{\frac{1}{p}}\left(\left(-t\right)^{p}+1\right)^{\frac{1}{p}}=-\left(\left(1-\alpha\right)^{\frac{1}{1-p}}-1\right)^{-\frac{1}{p}}.
\]
The optimal constant in~\eqref{eq: Normequi1} is thus given by $1\wedge\,t^{*}+\left(\frac{1+\left(-t^{*}\right)^{p}}{1-\alpha}\right)^{\frac{1}{p}}$,
i.e., $C=1\wedge\left(\left(1-\alpha\right)^{\frac{1}{1-p}}-1\right)^{\frac{p-1}{p}}$.
\end{proof}
\begin{thm}[Comparison with $L^{\infty}$]
\label{thm:Normequi2}For $\alpha\in(0,1)$ and $p<1$, the norms
$\EVaR_{\alpha}^{p}\left(\left|\cdot\right|\right)$ and $\left\Vert \cdot\right\Vert _{\infty}$
are equivalent. Indeed, we have that 
\begin{equation}
c\left\Vert \cdot\right\Vert _{\infty}\le\EVaR_{\alpha}^{p}\left(\left|\cdot\right|\right)\le\left\Vert \cdot\right\Vert _{\infty},\label{eq: Normequi2}
\end{equation}
where the constant $c=1-\left(1-\alpha\right){}^{-\frac{1}{p}}$ is
sharp. 
\end{thm}

\begin{proof}
By Theorem~\ref{thm:RVaR(p<1)} it is sufficient to consider the
case $p<0$. Then, by Theorem~\ref{thm:Monotonicity}, $\EVaR_{\alpha}^{p}\left(\left|Y\right|\right)\leq\left\Vert Y\right\Vert _{\infty}$
for all $Y\in L^{\infty}$. 

Without loss of generality let $Y\geq0$ and $\left\Vert Y\right\Vert _{\infty}=1$.
Let $t^{*}$ be the minimizer from the infimum representation~\eqref{eq:dualNorm-p>1}.
By Proposition~\ref{prop:objective} we know that $t^{*}\geq\left\Vert Y\right\Vert _{\infty}=1$
and thus 
\[
t^{*}=\left\Vert t^{*}\right\Vert _{p}\geq\left\Vert t^{*}-Y\right\Vert _{p}.
\]
It follows that 
\begin{align*}
\EVaR_{\alpha}^{p}\left(\left|Y\right|\right) & =\inf_{t>\esssup(Y)}\left\{ t-\left(\frac{1}{1-\alpha}\right)^{\nicefrac{1}{p}}\cdot\left\Vert t-Y\right\Vert _{p}\right\} \geq t^{*}-\left(\frac{1}{1-\alpha}\right)^{\nicefrac{1}{p}}t^{*}\\
 & =t^{*}\left(1-\left(1-\alpha\right){}^{-\frac{1}{p}}\right)\geq1-\left(1-\alpha\right){}^{-\frac{1}{p}}.
\end{align*}
We demonstrate that this constant is optimal for~\eqref{eq: Normequi2}.
Indeed, for every $\varepsilon>0$ consider the random variable $Y_{A}\coloneqq\one_{A}$
where $A$ is chosen such that $\varepsilon=P\left(A\right)$ for
which 
\begin{align*}
\EVaR_{\alpha}^{p}\left(Y_{A}\right) & =\inf_{t>1}\left\{ t-\left(\frac{1}{1-\alpha}\right)^{\frac{1}{p}}\left(\left(1-\varepsilon\right)\left(t\right)^{p}+\varepsilon\left(t-1\right)^{p}\right)^{\frac{1}{p}}\right\} .
\end{align*}
Similar to the previous proof we see that the right hand side is increasing
in $\varepsilon$ for each $t>1$ as 
\[
\frac{\mathrm{d}}{\mathrm{d}\varepsilon}\left(t-\left(\frac{1}{1-\alpha}\right)^{\frac{1}{p}}\left(\left(1-\varepsilon\right)\left(t\right)^{p}+\varepsilon\left(t-1\right)^{p}\right)^{\frac{1}{p}}\right)=-\left(\frac{1}{1-\alpha}\right)^{\frac{1}{p}}\left(-\left(t\right)^{p}+\left(t-1\right)^{p}\right)^{\frac{1}{p}}>0.
\]
Fixing a $t>1$ we can evaluate the limit of the right hand side 
\[
\lim_{\varepsilon\to0}\,\left\{ t-\left(\frac{1}{1-\alpha}\right)^{\frac{1}{p}}\left(\left(1-\varepsilon\right)\left(t\right)^{p}+\varepsilon\left(t-1\right)^{p}\right)^{\frac{1}{p}}\right\} =t-\left(\frac{1}{1-\alpha}\right)^{\frac{1}{p}}t=t\left(1-\left(1-\alpha\right)^{-\frac{1}{p}}\right).
\]
Taking the infimum over all feasible $t>1$ reveals that $c=1-(1-\alpha)^{-\frac{1}{p}}$
is the optimal constant in~\eqref{eq: Normequi2}. 
\end{proof}
We investigate the $\EVaR$ for different confidence levels $\alpha$.
It follows from~\eqref{eq:3-2} that $\EVaR_{\alpha}^{p}(Y)\le\EVaR_{\alpha^{\prime}}^{p}(Y)$
whenever $\alpha\le\alpha^{\prime}$. In addition, the following holds
true for nonnegative $Y\ge0$.
\begin{cor}[Comparison for different risk levels]
Let $\alpha\ge\alpha^{\prime}$, $p>1$ and $Y\ge0$. Then it holds
that
\[
\EVaR_{\alpha}^{p}(Y)\le\left(\left(\frac{1-\alpha}{1-\alpha^{\prime}}\right)^{\frac{1}{p-1}}-\left(\frac{1}{1-\alpha}\right)^{\frac{1}{1-p}}\right)^{\frac{1-p}{p}}\EVaR_{\alpha^{\prime}}^{p}(Y)
\]
and for $p<0$,
\[
\EVaR_{\alpha}^{p}(Y)\le\left(1-\left(1-\alpha^{\prime}\right){}^{-\frac{1}{p}}\right)^{-1}\EVaR_{\alpha^{\prime}}^{p}(Y).
\]
\end{cor}

\begin{proof}
The proof is an application of Theorem~\ref{thm:Normequi1} and Theorem~\ref{thm:Normequi2},
respectively. Therefore the $\EVaR_{\alpha}^{p}$ norms are equivalent
for different confidence levels. 
\end{proof}

\section{\label{sec:Dual-Norms}Dual norms}

As mentioned in the introduction (cf.~\eqref{eq:7-3}) every coherent
risk measure induces a semi-norm if applied to the absolute value
of the argument. We have already seen that 
\[
\left\Vert \cdot\right\Vert :=\EVaR_{\alpha}^{p}\left(\left|\cdot\right|\right)
\]
is in fact a norm on $L^{p}$ ( $L^{\infty}$ for $p<1$), respectively.
For these new norms we consider the associated dual norm on the dual
space $L^{p^{\prime}}$ given by 
\begin{equation}
\left\Vert Z\right\Vert _{\alpha,p^{\prime}}^{*}:=\sup_{\EVaR_{\alpha}^{p}(\left|Y\right|)\leq1}\E YZ.\label{eq: DualNorm-Def1}
\end{equation}
In what follows we give explicit representations for the dual norm
$\left\Vert \cdot\right\Vert _{\alpha,p^{\prime}}^{*}$. Further,
we describe the dual variables for which the maximum in~\eqref{eq: DualNorm-Def1}
is attained (if available). 

\subsection{Characterization of the dual norm}

The evaluation of the dual norm in~\eqref{eq: DualNorm-Def1} requires
computing the supremum over an infinite dimensional space of random
variables. Using the dual representations developed in Section~\ref{sec:Duality}
we can give an equivalent representation of those dual norms as a
supremum over real numbers, thus facilitating the evaluation of those
norms. 
\begin{cor}[Corollary to Theorem~\ref{thm:Normequi1}]
The norms $\left\Vert \cdot\right\Vert _{p^{\prime}}$ and $\left\Vert Z\right\Vert _{\alpha,p^{\prime}}^{*}$
are equivalent, it holds that 
\[
(1-\alpha)^{\frac{p^{\prime}-1}{p^{\prime}}}\left\Vert Z\right\Vert _{p^{\prime}}\le\left\Vert Z\right\Vert _{\alpha,p^{\prime}}^{*}\le\frac{1}{C}\left\Vert Z\right\Vert _{p^{\prime}},
\]
where $C$ is the constant given in Theorem~\ref{thm:Normequi1}.
The inequalities are sharp.
\end{cor}

\begin{prop}[The explicit dual norm for $p>1$]
\label{prop: DualNorm1}Let $p>1$ and $p^{\prime}=\frac{p}{p-1}.$
The dual norm $\left\Vert \cdot\right\Vert _{\alpha,p^{\prime}}^{*}$
of $\EVaR_{\alpha}^{p}\left(\left|\cdot\right|\right)$ is given by
\begin{equation}
\left\Vert Z\right\Vert _{\alpha,p^{\prime}}^{*}=\sup_{t\in\mathbb{R}}\,\frac{\E\left(t+\left|Z\right|^{p^{\prime}-1}\right)_{+}\left|Z\right|}{t+\left(\frac{1}{1-\alpha}\right)^{\frac{1}{p}}\left\Vert \left(t+\left|Z\right|^{q-1}\right)_{+}-t\right\Vert _{p}}.\label{eq:dualNorm-p>1}
\end{equation}
\end{prop}

\begin{proof}
We may assume that $Z\ge0$ so that we may restrict~\eqref{eq: DualNorm-Def1}
to $Y\ge0$. Observe first that $\lambda\geq\left\Vert Z\right\Vert _{\alpha,p^{\prime}}^{*}$
is equivalent to $0\geq\E YZ-\lambda\EVaR_{\alpha}^{p}\left(\left|Y\right|\right)$
for all $Y\geq0$. We maximize this expression with respect to $Y.$
The Lagrangian of this maximization problem is 
\begin{align*}
L(Y,\lambda,\mu) & =\E YZ-\lambda\EVaR_{\alpha}^{p}\left(\left|Y\right|\right)-\E Y\mu,
\end{align*}
where $\mu$ is the Lagrange multiplier associated to the constraint
$Y\geq0$. The Lagrangian at the optimal $Y$ with optimizer~$t^{*}$
for $\EVaR_{\alpha}^{p}(Y)$ is 
\begin{align*}
L(Y,\lambda,\mu) & =\E YZ-\lambda t^{*}-\lambda\left(\frac{1}{1-\alpha}\right)^{^{\frac{1}{p}}}\left\Vert \left(Y-t^{*}\right)_{+}\right\Vert _{p}-\E Y\mu.
\end{align*}
The directional derivative of the Lagrangian in direction $H\in L^{p}$
at $Y$ is 
\[
\frac{\partial}{\partial Y}L(Y,\lambda,\mu)H=\E HZ-\lambda\left(\frac{1}{1-\alpha}\right)^{^{\frac{1}{p}}}\E\left[\left(Y-t^{*}\right)_{+}^{p}\right]^{\frac{1}{p}-1}\cdot\E H\left(Y-t^{*}\right)_{+}^{p-1}-\E H\mu.
\]
The derivative vanishes in every direction $H$ so that 
\[
Z-\mu=c\left(Y-t^{*}\right)_{+}^{p-1},
\]
where $c=\left(\frac{1}{1-\alpha}\right)^{^{\frac{1}{p}}}\E\left[\left(Y-t^{*}\right)_{+}^{p}\right]^{\frac{1}{p}-1}>0$.
By complimentary slackness for the optimal $Y$ and $\mu$, 
\[
Y>0\iff\mu=0\iff Z=c\left(Y-t^{*}\right)_{+}^{p-1}>c\left(-t^{*}\right)_{+}^{p-1},
\]
which is equivalent to $Y=\left(t^{*}+\left(\frac{Z}{c}\right)^{p^{\prime}-1}\right)_{+}$.
Denote the optimal $Y$ in~\eqref{eq: DualNorm-Def1} by $Y_{\max}$
with optimizer $t^{*}$ of $\EVaR_{\alpha}^{p}(Y_{\max})$. Then the
above consideration implies that
\begin{align*}
\sup_{Y\neq0}\frac{\E YZ}{\EVaR_{\alpha}^{p}\left(\left|Y\right|\right)} & =\frac{\E Y_{\max}Z}{\EVaR_{\alpha}^{p}\left(\left|Y_{\max}\right|\right)}\\
 & =\frac{\E\left(t^{*}+\left(\frac{Z}{c}\right)^{p^{\prime}-1}\right)_{+}Z}{t^{*}+\left(\frac{1}{1-\alpha}\right)^{\frac{1}{p}}\left\Vert \left(\left(t^{*}+\left(\frac{Z}{c}\right)^{p^{\prime}-1}\right)_{+}-t^{*}\right)_{+}\right\Vert _{p}}\\
 & =\frac{\left(\frac{1}{c}\right)^{q-1}\E\left(c^{q-1}t^{*}+Z^{p^{\prime}-1}\right)_{+}Z}{t^{*}+\left(\frac{1}{c}\right)^{q-1}\left(\frac{1}{1-\alpha}\right)^{\frac{1}{p}}\left\Vert \left(\left(c^{q-1}t^{*}+Z^{q-1}\right)_{+}-c^{q-1}t^{*}\right)_{+}\right\Vert _{p}}.
\end{align*}
We assumed that $Z\geq0$ and hence $((t+Z^{p^{\prime}-1})_{+}-t)_{+}=(Z^{p^{\prime}-1}+t)_{+}-t$.
The dual norm then simplifies to
\begin{align*}
\sup_{Y\neq0}\frac{\E YZ}{\EVaR_{\alpha}^{p}\left(\left|Y\right|\right)} & =\frac{\E\left(c^{q-1}t^{*}+Z^{p^{\prime}-1}\right)_{+}Z}{c^{q-1}t^{*}+\left(\frac{1}{1-\alpha}\right)^{\frac{1}{p}}\left\Vert \left(c^{q-1}t^{*}+Z^{q-1}\right)_{+}-c^{q-1}t^{*}\right\Vert _{p}}\\
 & =\sup_{t\in\mathbb{R}}\,\frac{\E\left(t+Z^{p^{\prime}-1}\right)_{+}Z}{t+\left(\frac{1}{1-\alpha}\right)^{\frac{1}{p}}\left\Vert \left(t+Z^{q-1}\right)_{+}-t\right\Vert _{p}},
\end{align*}
which concludes the proof.
\end{proof}
In the case $p<0$ we deduce a similar result to Proposition~\ref{prop: DualNorm1}.
We can give the following characterization, which is again a supremum
over one single parameter. 
\begin{prop}[Explicit dual norm for $p<0$]
\label{prop:Dualnorm2}For $p<0$, the dual norm of $\left\Vert \cdot\right\Vert =\EVaR_{\alpha}^{p}\left(\left|\cdot\right|\right)$
is given by 
\begin{equation}
\left\Vert Z\right\Vert _{\alpha,p^{\prime}}^{*}=\sup_{t>\esssup\left(Z^{p^{\prime}-1}\right)}\ \frac{\E\left(t-\left|Z\right|^{p^{\prime}-1}\right)\left|Z\right|}{t-\left(\frac{1}{1-\alpha}\right)^{\frac{1}{p}}\left\Vert \left|Z\right|^{p^{\prime}-1}\right\Vert _{p}}.\label{eq:DualNorm2}
\end{equation}
\end{prop}

\begin{proof}
We may assume that $Z\ge0$ so that we may restrict~\eqref{eq: DualNorm-Def1}
to $Y\ge0$. Observe first that that $\lambda\geq\left\Vert Z\right\Vert _{\alpha,p^{\prime}}^{*}$
is equivalent to $0\geq\E YZ-\lambda\EVaR_{\alpha}^{p}\left(\left|Y\right|\right)$
for all $Y\geq0$. We maximize this expression with respect to $Y$.
Then the Lagrange formulation of this maximization problem is 
\begin{align*}
L(Y,\lambda,\mu) & =\E YZ-\lambda\EVaR_{\alpha}^{p}\left(\left|Y\right|\right)-\E Y\mu\\
 & =\E YZ-\lambda t^{*}+\lambda\left(\frac{1}{1-\alpha}\right)^{^{\frac{1}{p}}}\left\Vert t^{*}-Y\right\Vert _{p}-\E Y\mu,
\end{align*}
where $t^{*}$ is the optimizer of $\EVaR_{\alpha}^{p}(Y)$ and $\mu$
the Lagrange multiplier associated to the constraint $Y\geq0$. $L$
is differentiable in each direction and the directional derivative
in $Y$ in the direction $H$ is 
\[
\frac{\partial}{\partial Y}L(Y,\lambda,\mu)H=\E HZ-\lambda\left(\frac{1}{1-\alpha}\right)^{^{\frac{1}{p}}}\E\left[\left(t^{*}-Y\right){}^{p}\right]^{\frac{1}{p}-1}\E H\left(t^{*}-Y\right)^{p-1}-\E H\mu.
\]
The derivative vanishes in every direction $H$ and consequently we
get 
\[
Z-\mu=\frac{1}{c}\left(t^{*}-Y\right)^{p-1}.
\]
Here, $\frac{1}{c}=\lambda\left(\frac{1}{1-\alpha}\right)^{^{\frac{1}{p}}}\E\left[\left(t^{*}-Y\right)^{p}\right]^{\frac{1}{p}-1}$.
By complimentary slackness for optimal $Y$ and $\mu$, it follows
that
\[
Y>0\iff\mu=0\iff Z=\frac{1}{c}\left(t^{*}-Y\right){}^{p-1}.
\]
This shows that it is enough to consider $Y$ of the form $Y=t-\left(cZ\right){}^{p^{\prime}-1}$
in the definition of the dual norm. Denote now the optimal $Y$ in~\eqref{eq: DualNorm-Def1}
by $Y_{\max}$. Then the above implies that for the optimizer $t^{*}$
of $\EVaR_{\alpha}^{p}\left(Y_{\max}\right)$ we have
\[
Y_{\max}=t^{*}-\left(cZ\right){}^{p^{\prime}-1}.
\]
And therefore the expression~\eqref{eq: DualNorm-Def1} reduces to
\begin{align*}
\sup_{Y\neq0}\frac{\E YZ}{\EVaR_{\alpha}^{p}\left(\left|Y\right|\right)} & =\frac{\E Y_{\max}Z}{\EVaR_{\alpha}^{p}\left(\left|Y_{\max}\right|\right)}=\frac{\E\left(t^{*}-\left(cZ\right){}^{p^{\prime}-1}\right)Z}{t^{*}-\left(\frac{1}{1-\alpha}\right)^{\frac{1}{p}}\left\Vert \left(cZ\right){}^{p^{\prime}-1}\right\Vert _{p}}\\
 & =\frac{\E\left[\left(\frac{t^{*}}{c^{p^{\prime}-1}}-Z^{p^{\prime}-1}\right)Z\right]}{\frac{t^{*}}{c^{p^{\prime}-1}}-\left(\frac{1}{1-\alpha}\right)^{\frac{1}{p}}\left\Vert Z^{p^{\prime}-1}\right\Vert _{p}}.
\end{align*}
Notice that $Y>0$ implies that $\frac{t^{*}}{c^{p^{\prime}-1}}>\esssup\left(Z^{p^{\prime}-1}\right)$
and thus 
\[
\frac{\E\left[\left(\frac{t^{*}}{c^{p^{\prime}-1}}-Z^{p^{\prime}-1}\right)Z\right]}{\frac{t^{*}}{c^{p^{\prime}-1}}-\left(\frac{1}{1-\alpha}\right)^{\frac{1}{p}}\left\Vert Z^{p^{\prime}-1}\right\Vert _{p}}=\sup_{t>Z^{p^{\prime}-1}}\,\frac{\E\left[\left(t-Z^{p^{\prime}-1}\right)Z\right]}{t-\left(\frac{1}{1-\alpha}\right)^{\frac{1}{p}}\left\Vert Z^{p^{\prime}-1}\right\Vert _{p}},
\]
which shows the assertion.
\end{proof}

\subsection{Hahn\textendash Banach functionals}

We now describe the Hahn-Banach functionals corresponding to $Y\in L^{p}$
and $Z\in L^{p^{\prime}}$explicitly. This means we identify the random
variable $Z\in L^{p^{\prime}}$which maximizes 
\[
\EVaR(|Y|)=\sup_{Z\neq0}\frac{\E YZ}{\left\Vert Z\right\Vert _{\alpha,p^{\prime}}^{*}}
\]
and the random variable $Y\in L^{p}$ which maximizes
\[
\left\Vert Z\right\Vert _{\alpha,p^{\prime}}^{*}=\sup_{Y\neq0}\frac{\E YZ}{\EVaR_{\alpha}^{p}(\left|Y\right|)}.
\]
\begin{prop}
\label{prop: HB1}For $p>1$ let $Y\in L^{p}$ and suppose that there
is an optimizer $t^{*}\in\mathbb{R}$ of 
\[
\inf_{t\in\mathbb{R}}\left\{ t+\left(\frac{1}{1-\alpha}\right)^{\nicefrac{1}{p}}\cdot\left\Vert (\left|Y\right|-t)_{+}\right\Vert _{p}\right\} .
\]
Then $Z^{\prime}:=\sign(Y)\cdot\left(\left|Y\right|-t^{*}\right){}_{+}^{p-1}$
maximizes the expression
\[
\EVaR_{\alpha}^{p}\left(\left|Y\right|\right)=\sup_{Z\neq0}\frac{\E YZ}{\left\Vert Z\right\Vert _{\alpha,p^{\prime}}^{*}}.
\]
\end{prop}

\begin{proof}
Without loss of generality we assume $Y\geq0$. Then, by the definition
of the dual norm, the definition of $Z^{\prime}$ and Eq.~\eqref{eq:7-4}
in Remark~\ref{rem:7} we have $\E Z^{\prime}\cdot\EVaR_{\alpha}^{p}\left(\left|Y\right|\right)=\E YZ^{\prime}$.
It is therefore enough to verify that $\E Z^{\prime}=\left\Vert Z^{\prime}\right\Vert _{\alpha,p^{\prime}}^{*}$.
Since $\frac{Z^{^{\prime}}}{\E Z^{\prime}}$ is a density we have
\[
\left\Vert Z^{\prime}\right\Vert _{\alpha,p^{\prime}}^{*}=\sup_{Y\neq0}\frac{\E YZ^{\prime}}{\EVaR_{\alpha}^{p}(Y)}=\frac{\E Y^{*}Z^{\prime}}{\E Y^{*}Z^{*}}=\E Z^{\prime},
\]
where $Y^{*}$ is the maximizer of the above supremum and $Z^{*}$
the optimal density for $Y^{*}$.
\end{proof}
We now address the converse question, which is: given $Z$, what is
the random variable $Y$ to achieve equality in~\eqref{eq: DualNorm-Def1}?
\begin{prop}
\label{prop: HB2}For $p^{\prime}>1$ let $Z\in L^{p^{\prime}}$and
suppose that there is an optimal $t^{*}\in\mathbb{R}$ in 
\[
\sup_{t\in\mathbb{R}}\frac{\E\left(t+\left|Z\right|^{p^{\prime}-1}\right)_{+}\left|Z\right|}{t+\left(\frac{1}{1-\alpha}\right)^{\frac{1}{p}}\left\Vert \left(t+\left|Z\right|^{p^{\prime}-1}\right)_{+}-t\right\Vert _{p}}.
\]
Then $Y^{\prime}:=\sign(Z)\cdot\left(t^{*}+\left|Z\right|^{p^{\prime}-1}\right)_{+}$
satisfies the equality
\[
\E Y^{\prime}Z=\EVaR_{\alpha}^{p}\left(Y^{\prime}\right)\cdot\left\Vert Z\right\Vert _{\alpha,p^{\prime}}^{*}.
\]
\end{prop}

\begin{proof}
Without loss of generality, we may assume that $Z\geq0$. By assumption
\[
\left\Vert Z\right\Vert _{\alpha,p^{\prime}}^{*}=\frac{\E\left(t^{*}+\left|Z\right|^{p^{\prime}-1}\right)_{+}\left|Z\right|}{t^{*}+\left(\frac{1}{1-\alpha}\right)^{\frac{1}{p}}\left\Vert \left(\left(t^{*}+\left|Z\right|^{p^{\prime}-1}\right)-t^{*}\right)\right\Vert _{p}}.
\]
Then 
\[
\E Y^{\prime}Z\leq\left\Vert Z\right\Vert _{\alpha,p^{\prime}}^{*}\EVaR_{\alpha}^{p}(Y^{\prime})\leq\left\Vert Z\right\Vert _{\alpha,p^{\prime}}^{*}\left(t^{*}+\left(\frac{1}{1-\alpha}\right)^{\frac{1}{p}}\left\Vert \left(Y^{\prime}-t^{*}\right)_{+}\right\Vert _{p}\right)=\E Y^{\prime}Z,
\]
hence equality holds and the assertion follows.
\end{proof}

We now derive the corresponding Hahn-Banach functionals for the dual
norm $\left\Vert Z\right\Vert _{\alpha,p}^{*}$ for $p<0$. The proofs
are analogous to the case $p>1$ in Proposition~\ref{prop: HB1}
(Proposition~\ref{prop: HB2}, resp.) and therefore are omitted.
\begin{prop}
\label{prop: HB3}Let $Y\in L^{\infty}$ and $p<0$. Further suppose
that there is an optimal $t^{*}\in\mathbb{R}$ in 
\[
\inf_{t>\esssup Y}\left\{ t-\left(\frac{1}{1-\alpha}\right)^{\nicefrac{1}{p}}\cdot\left\Vert t-\left|Y\right|\right\Vert _{p}\right\} .
\]
Then $Z^{\prime}:=\sign(Y)\cdot\left(t^{*}-\left|Y\right|\right){}^{p-1}$
maximizes the expression
\[
\EVaR_{\alpha}^{p}\left(\left|Y\right|\right)=\sup_{Z\neq0}\frac{\E YZ}{\left\Vert Z\right\Vert _{\alpha,p^{\prime}}^{*}}.
\]
\end{prop}

\begin{prop}
\label{prop: HB4}Let $Z\neq0$ and $p<0$ and suppose that there
is an optimal $t^{*}<\infty$ in 
\[
\sup_{t>\esssup Z^{p^{\prime}-1}}\,\frac{\E\left(t-\left|Z\right|^{p^{\prime}-1}\right)\left|Z\right|}{t-\left(\frac{1}{1-\alpha}\right)^{\frac{1}{p}}\left\Vert \left|Z\right|^{p^{\prime}-1}\right\Vert _{p}}.
\]
Then $Y^{\prime}:=\sign(Z)\cdot\left(t^{*}-\left|Z\right|^{p^{\prime}-1}\right)$
satisfies the equality
\[
\left\Vert Z\right\Vert _{\alpha,p^{\prime}}^{*}\cdot\EVaR_{\alpha}^{p}\left(Y^{\prime}\right)=\E Y^{\prime}Z.
\]
\end{prop}

With the maximizer of the Hahn-Banach functionals at hand, we can
give an alternative supremum representation of the Entropic Value-at-Risk
based on Rényi entropy.
\begin{cor}[Dual representation of $\EVaR$]
The Entropic Value-at-Risk has the alternative dual representation
\[
\EVaR_{\alpha}^{p}(Y)=\sup\left\{ \E YZ,\ Z\geq0,\ \E Z=1,\ \left\Vert Z\right\Vert _{\alpha,p^{\prime}}^{*}\leq1\right\} ,
\]
where either $p>1$ or $p<0$. 
\end{cor}

\begin{proof}
Suppose that $p>1$ first. Then, by Proposition~\ref{prop: DualNorm1},
the inequality $\left\Vert Z\right\Vert _{\alpha,p^{\prime}}^{*}\leq1$
is equivalent to 
\[
\frac{\E\left(t+\left|Z\right|^{p^{\prime}-1}\right)_{+}\left|Z\right|}{t+\left(\frac{1}{1-\alpha}\right)^{\frac{1}{p}}\left\Vert \left(t+\left|Z\right|^{p^{\prime}-1}\right)_{+}-t\right\Vert _{p}}\leq1
\]
for all $t\in\mathbb{R}$. Setting $t=0$ it follows that 
\[
\E Z^{p^{\prime}}\leq\left(\frac{1}{1-\alpha}\right)^{\frac{1}{p}}\left(\E Z^{p^{\prime}}\right)^{\frac{1}{p}},
\]
which is equivalent to $\left\Vert Z\right\Vert _{p^{\prime}}\leq\left(\frac{1}{1-\alpha}\right)^{\frac{1}{p}}$
and so $\sup\left\{ \E YZ,\ Z\geq0,\ \E Z=1,\ \left\Vert Z\right\Vert _{\alpha,p^{\prime}}^{*}\leq1\right\} \leq\EVaR_{\alpha}^{p}(Y)$. 

For $p<0$ the constraint $\left\Vert Z\right\Vert _{\alpha,p^{\prime}}^{*}\leq1$
is equivalent to 
\[
\frac{\E\left(t-\left|Z\right|^{p^{\prime}-1}\right)\left|Z\right|}{t-\left(\frac{1}{1-\alpha}\right)^{\frac{1}{p}}\left\Vert \left|Z\right|^{p^{\prime}-1}\right\Vert _{p}}<1,
\]
which under the assumptions $Z\geq0$ and $\E Z=1$ can be rewritten
as $\E Z^{p^{\prime}}\geq\left(\frac{1}{1-\alpha}\right)^{\frac{1}{p}}\left\Vert Z^{p^{\prime}-1}\right\Vert _{p}$or
$\left\Vert Z\right\Vert _{p^{\prime}}\geq\left(\frac{1}{1-\alpha}\right)^{\frac{p^{\prime}-1}{p^{\prime}}}$. 

It remains to be shown that the maximizing $Z$ in~\eqref{eq: EVaRDef1}
satisfies the constraint $\left\Vert Z\right\Vert _{\alpha,p^{\prime}}^{*}\leq1$.
In fact, by Theorem~\ref{thm:primal1}, we know that for $p>1$ the
optimal $Z$ in $\EVaR_{\alpha}^{p}(Y)$ is given by $Z=\frac{(Y-t^{*})_{+}^{p-1}}{\E(Y-t^{*})_{+}^{p-1}}$,
where $t^{*}$ is the optimizer of~\eqref{eq: EVaRinf1}. By Proposition~\ref{prop: HB2},
the random variable $Y^{\prime}:=\sign(Z)\cdot\left(Z^{p^{\prime}-1}-t^{*}\right)$
satisfies 
\[
\E Y^{\prime}Z=\EVaR_{\alpha}^{p}\left(Y^{\prime}\right)\ \left\Vert Z\right\Vert _{\alpha,p^{\prime}}^{*}
\]
and since $Z$ is feasible for~\eqref{eq: EVaRDef1} it follows that$\left\Vert Z\right\Vert _{\alpha,p^{\prime}}^{*}=\frac{\E Y^{\prime}Z}{\EVaR_{\alpha}^{p}\left(Y^{\prime}\right)}\leq1$
by definition of $\EVaR_{\alpha}^{p}$. The same reasoning applies
to the case $p<0$. Here, the optimal $Z$ is given by $Z=\frac{(t^{*}-Y)^{p-1}}{\E(t^{*}-Y)^{p-1}}$
according to Theorem~\ref{thm:primal2}. We apply Proposition~\ref{prop: HB4}
to conclude the assertion.
\end{proof}

\subsection{Kusuoka representation}

The $\EVaR_{\alpha}^{p}$  is a version independent coherent risk
measure for which consequently a Kusuoka representation can be obtained
(cf.\ \citet{Kusuoka-R}). We derive the Kusuoka representation from
its dual representation. 
\begin{prop}[Kusuoka representation]
The Kusuoka representation of the Entropic Value-at-Risk for $p>1$
or $p<0,\,\alpha\in\left[0,1\right)$ and $Y\in L^{p}$ ($L^{\infty})$,
respectively is 
\begin{equation}
\EVaR_{\alpha}^{p}(Y)=\sup_{\mu}\int_{0}^{1}\AVaR_{x}(Y)\ \mu\left(\mathrm{d}x\right),\label{eq:Kusuoka}
\end{equation}
where the supremum is among all probability measures $\mu$ on $[0,1)$
for which the function 
\[
\sigma_{\mu}\left(u\right)=\int_{0}^{u}\frac{1}{1-v}\mu(\mathrm{d}v)
\]
satisfies
\[
\int_{0}^{1}\sigma_{\mu}\left(u\right)^{p^{\prime}}\mathrm{d}u\leq\left(\frac{1}{1-\alpha}\right)^{p^{\prime}-1}.
\]
The supremum in~\eqref{eq:Kusuoka} is attained for the measure $\mu_{\sigma*}$
associated with the distortion function 
\[
\sigma^{*}\left(u\right)\coloneqq F_{Z^{*}}^{-1}\left(u\right)=\VaR_{u}\left(Z^{*}\right),
\]
where $Z^{*}$ is the optimal random variable in~\eqref{eq: EVaRDef1}
and $F_{Z^{*}}^{-1}(u)=\VaR_{u}(Z^{*})$ its generalized inverse. 
\end{prop}

\begin{proof}
The representation follows from the supremum representation~\eqref{eq: EVaRDef1}.
Observe that the supremum is attained and for the maximizing density
$Z^{*}$ in\textbf{~\eqref{eq: EVaRDef1}}
\[
\E\left(Z^{*}\right)^{p^{\prime}}=\left(\frac{1}{1-\alpha}\right)^{p^{\prime}-1}
\]
holds. By definition of $\sigma^{*}$ we have 
\[
\E\left(Z^{*}\right)^{p^{\prime}}=\int_{0}^{1}\left(\sigma^{*}\left(u\right)\right)^{p^{\prime}}\mathrm{d}u.
\]
We define the measure $\mu^{*}\left(A\right)\coloneqq\sigma^{*}\left(0\right)\cdot\delta_{0}\left(A\right)+\int_{A}\left(1-u\right)\mathrm{d}\sigma^{*}\left(u\right)$
for a measurable set $A\subseteq\left[0,1\right)$. For this measure
we have
\[
\int_{0}^{u}\frac{1}{1-x}\mu^{*}\left(\mathrm{d}x\right)=\sigma^{*}\left(0\right)+\int_{0}^{u}\frac{1}{1-v}\left(1-v\right)\mathrm{d}\sigma^{*}\left(v\right)=\sigma^{*}\left(u\right)
\]
and therefore $\sigma^{*}$ is feasible in the above supremum. Furthermore
\begin{align*}
\int_{0}^{1}\AVaR_{x}(Y)\mu^{*}\left(\mathrm{d}x\right) & =\sigma^{*}\left(0\right)\AVaR_{0}(Y)+\int_{0}^{1}\frac{1}{1-x}\int_{x}^{1}F_{Y}^{-1}\left(u\right)\mathrm{d}u\,\left(1-x\right)\mathrm{d}\sigma^{*}\left(x\right)\\
 & =\int_{0}^{1}\sigma^{*}\left(u\right)\,F_{Y}^{-1}\left(u\right)\,\mathrm{d}u\\
 & =\E YZ^{*}=\EVaR_{\alpha}^{p}(Y)
\end{align*}
and hence the assertion follows.
\end{proof}
In this section we derive a computational convenient representation
of the $\EVaR$-dual norms and gave the explicit formulas for their
maximizers. This allowed us to give another supremum representation
for $\EVaR$. We further elaborated on the Kusuoka representation
of $\EVaR$.

\section{\label{sec:Summary}Concluding Remarks}

This paper introduces entropic risk measures specified by a family
of entropies. These risk measures are interesting in stochastic optimization
and its applications as entropy allows the interpretation of information
losses and the corresponding risk measures reflect ambiguity in terms
of lost information. 

\citet{AhmadiJavidEVaR-R} introduces the so-called Entropic Value-at-Risk
as the tightest upper bound for the Value-at-Risk and the Average-Value-at-Risk
using Chernoff's inequality. He expresses the classical Entropic Value-at-Risk
by employing Shannon entropy.

We extend this work to the class of Rényi entropies and we show that
the associated risk measures are monotone and continuous with respect
to the Rényi order. The case studied in \citet{AhmadiJavidEVaR-R}
arises as a special, limiting case. For the Rényi entropy of order
larger than $1$, the Rényi entropic risk measures interpolate the
Average Value-at-Risk and the Entropic Value-at-Risk based on Shannon
entropy. For the Rényi order smaller than $1$ the Entropic Value-at-Risk
based on Rényi entropy dominates the Entropic Value-at-Risk. The essential
supremum is recovered as a limiting case as well.

Most importantly from a viewpoint of stochastic optimization we derive
an equivalent infimum representation of the risk measures~\eqref{eq: EVaRDef1},
where the infimum is considered over a real variable. This allows
a efficient computation of stochastic programs employing these risk
measures based on Rényi entropy.

We further study the norms associated with entropic risk measures
and elaborate the exact constants in comparing them with Hölder norms.
In this way we relate them to higher order risk measures. We further
explicit the formulas of the dual norms and the corresponding Hahn\textendash Banach
functionals. We use the duality results to derive alternative dual
and Kusuoka representations and state the maximizing densities explicitly. 

\section{Acknowledgment}

We wish to thank the editor and the referees of this journal for their
time and commitment in assessing the paper.\bibliographystyle{abbrvnat}
\bibliography{LiteraturRuben}

\appendix

\section{\label{sec:Appendix}Appendix }

We give a proof of Theorem~\ref{thm:log-convex}. For this we recall
the following result first.
\begin{thm}[Envelope Theorem]
Let $f(x,q)$ be a continuously differentiable function with $x\in\mathbb{R}$
and $q\in\mathbb{R}$. Assume that the parametric problem 
\[
v(q):=\max_{x\in\mathbb{R}}\,f(x,q),
\]
admits a continuously differentiable solution $x^{*}(q)$. Then the
optimal value function $v(q)=f(x^{*}(q),q)$ of $f$ has the derivative
\[
\frac{\mathrm{d}v}{\mathrm{d}q}(q)=\frac{\mathrm{d}f}{\mathrm{d}q}(x^{*}(q),q).
\]
 
\end{thm}

Before we give a proof of Theorem~\ref{thm:log-convex} we show that
$\EVaR_{\alpha}^{p}(Y)$ is convex in its dual order. 
\begin{lem}[Convexity of the Entropic Value-at-Risk]
\label{lem:convexEVaR}For $1<p_{0},\,p_{1}$ and $0\leq\lambda\leq1$
define $p_{\lambda}^{\prime}:=(1-\lambda)p_{0}^{\prime}+\lambda p_{1}^{\prime}$,
where $p_{0}^{\prime}$ ($p_{1}^{\prime}$, resp.) is the Hölder conjugate
exponent of $p_{0}$ ($p_{1}$, resp.). Then, for $Y\in L^{\infty}$,
$\alpha\in(0,1)$ and $p_{\lambda}:=\frac{p_{\lambda}^{\prime}}{p_{\lambda}^{\prime}-1}$
we have that 
\[
\EVaR_{\alpha}^{p_{\lambda}}\left(\left|Y\right|\right)\leq\left(1-\lambda\right)\EVaR_{\alpha}^{p_{0}}\left(\left|Y\right|\right)+\lambda\EVaR_{\alpha}^{p_{1}}\left(\left|Y\right|\right).
\]
This means the Entropic Value-at-Risk is convex in its conjugate order. 
\end{lem}

\begin{rem}
The following proof of the preceding lemma is rather technical. For
this we describe the procedure in brief first. Using the envelope
theorem we can calculate the $p^{\prime}$-derivative of the infimum
representation~\eqref{eq: EVaRinf1} at the optimal point. Using
the relationship between the optimizer of the infimum and supremum
representations we derive a more useful formula for this derivative.
We conclude the proof by showing that the $p^{\prime}$-derivative
of $\EVaR_{\alpha}^{p}(Y)$ is increasing.
\end{rem}

The next lemma shows that the optimizer of the infimum representation~\eqref{eq: EVaRinf1}
of $\EVaR_{\alpha}^{p}(Y)$ is nondecreasing. The main result then
follows as a simple corollary. 
\begin{proof}[Proof of Lemma~\ref{lem:convexEVaR}]
Let $\alpha\in(0,1)$, $\beta=\frac{1}{1-\alpha}$ and $Y\in L^{\infty}$.
Without loss of generality we may  restrict ourselves to $Y>0$ and
further suppose that $Y$ is not constant. We apply the envelope theorem
to the infimum representation~\eqref{eq: EVaRinf1} of the Entropic
Value-at-Risk (see Theorem~\ref{thm:primal1}). The product rule
yields 
\[
\frac{\mathrm{d}}{\mathrm{d}p^{\prime}}\EVaR_{\alpha}^{p}(Y)=\left(\frac{\mathrm{d}}{\mathrm{d}p^{\prime}}\beta^{\frac{p^{\prime}-1}{p^{\prime}}}\right)\left\Vert \left(Y-x\right)_{+}\right\Vert _{\frac{p^{\prime}}{p^{\prime}-1}}+\beta^{\frac{p^{\prime}-1}{p^{\prime}}}\,\frac{\mathrm{d}}{\mathrm{d}p^{\prime}}\left\Vert \left(Y-x\right)_{+}\right\Vert _{\frac{p^{\prime}}{p^{\prime}-1}}\bigg|_{x=x^{*}(p^{\prime})},
\]
where $p^{\prime}$ is the Hölder conjugate of $p>1$. For the remainder
of the proof we set $x=x^{*}(p^{\prime})$ as it is clear which $p^{\prime}$
is considered. We now detail the above derivative as 
\[
\frac{\mathrm{d}}{\mathrm{d}p^{\prime}}\beta^{\frac{p^{\prime}-1}{p^{\prime}}}=\left(\frac{1}{p^{\prime}}\right)^{2}\beta^{\frac{p^{\prime}-1}{p^{\prime}}}\log\beta
\]
 and 
\begin{align*}
\frac{\mathrm{d}}{\mathrm{d}p^{\prime}}\left\Vert \left(Y-x\right)_{+}\right\Vert _{\frac{p^{\prime}}{p^{\prime}-1}} & =\left\Vert \left(Y-x\right)_{+}\right\Vert _{\frac{p^{\prime}}{p^{\prime}-1}}\cdot\left(\frac{1}{p^{\prime^{2}}}\log\E\left(Y-x\right)_{+}^{\frac{p^{\prime}}{p^{\prime}-1}}-\frac{1}{(p^{\prime}-1)p^{\prime}}\frac{\E\left(Y-x\right)_{+}^{\frac{p^{\prime}}{p^{\prime}-1}}\log\left(Y-x\right)_{+}}{\E\left(Y-x\right)_{+}^{\frac{p^{\prime}}{p^{\prime}-1}}}\right).
\end{align*}
Therefore the derivative of $\EVaR_{\alpha}^{p}$ rewrites as
\begin{align}
\frac{\mathrm{d}}{\mathrm{d}p^{\prime}}\EVaR_{\alpha}^{p}(Y)= & \beta^{\frac{p^{\prime}-1}{p^{\prime}}}\left\Vert \left(Y-x\right)_{+}\right\Vert _{\frac{p^{\prime}}{p^{\prime}-1}}\cdot\nonumber \\
 & \left(\frac{1}{p^{\prime^{2}}}\left(\log\beta+\log\E\left(Y-x\right)_{+}^{\frac{p^{\prime}}{p^{\prime}-1}}\right)-\frac{1}{(p^{\prime}-1)p^{\prime}}\cdot\frac{\E\left(Y-x\right)_{+}^{\frac{p^{\prime}}{p^{\prime}-1}}\log\left(Y-x\right)_{+}}{\E\left(Y-x\right)_{+}^{\frac{p^{\prime}}{p^{\prime}-1}}}\right).\label{eq:derivative RVaR}
\end{align}
We continue by rewriting the last factor. To this end it is useful
to consider the maximizing densities of the supremum representation
of $\EVaR$, cf.~\eqref{eq:7-4},
\begin{equation}
Z=Z^{*}(p^{\prime})=\frac{\left(Y-x^{*}(p^{\prime})\right)_{+}^{\frac{1}{p^{\prime}-1}}}{\E\left(Y-x^{*}(p^{\prime})\right)_{+}^{\frac{1}{p^{\prime}-1}}},\label{eq:18}
\end{equation}
from which the identity 
\begin{equation}
\E\left(Y-x\right)_{+}^{\frac{p^{\prime}}{p^{\prime}-1}}=\E Z^{p^{\prime}}\cdot\left(\E\left(Y-x\right)_{+}^{\frac{1}{p^{\prime}-1}}\right)^{p^{\prime}}\label{eq:17}
\end{equation}
follows. Furthermore, by the convexity of the set of feasible densities
in~\eqref{eq: EVaRDef1}, each optimal $Z=Z^{*}(p^{\prime})$ satisfies
the identity $\frac{1}{p^{\prime}-1}\log\E Z^{p^{\prime}}=\log\beta$.
For the second factor of~\eqref{eq:derivative RVaR} we have 
\begin{align}
\frac{1}{p^{\prime^{2}}}\left(\log\beta+\log\E\left(Y-x\right)_{+}^{\frac{p^{\prime}}{p^{\prime}-1}}\right) & =\frac{1}{p^{\prime}}\left(\log\beta+\log\E\left(Y-x\right)_{+}^{\frac{1}{p^{\prime}-1}}\right),\label{eq:19}
\end{align}
where we have used~\eqref{eq:17} and, by employing~\eqref{eq:18},
\begin{align}
\frac{\E\left(Y-x\right)_{+}^{\frac{p^{\prime}}{p^{\prime}-1}}\log\left(Y-x\right)_{+}}{\E\left(Y-x\right)_{+}^{\frac{p^{\prime}}{p^{\prime}-1}}} & =\frac{(p^{\prime}-1)\E Z^{p^{\prime}}\log Z}{\beta^{p^{\prime}-1}}+(p^{\prime}-1)\log\E\left(Y-x\right)_{+}^{\frac{1}{p^{\prime}-1}}.\label{eq:20}
\end{align}
In conclusion, the factor $\frac{1}{p^{\prime^{2}}}\left(\log\beta+\log\E\left(Y-x\right)_{+}^{\frac{p^{\prime}}{p^{\prime}-1}}\right)-\frac{1}{(p^{\prime}-1)p^{\prime}}\cdot\frac{\E\left(Y-x\right)_{+}^{\frac{p^{\prime}}{p^{\prime}-1}}\log\left(Y-x\right)}{\E\left(Y-x\right)_{+}^{\frac{p^{\prime}}{p^{\prime}-1}}}$
in~\eqref{eq:derivative RVaR} is (cf.~\eqref{eq:19} and~\eqref{eq:20})
\[
\frac{1}{p^{\prime}}\left(\log\beta+\log\E\left(Y-x\right)_{+}^{\frac{1}{p^{\prime}-1}}\right)-\frac{\E Z^{p^{\prime}}\log Z}{p^{\prime}\beta^{p^{\prime}-1}}-\frac{1}{p^{\prime}}\log\E\left(Y-x\right)_{+}^{\frac{1}{p^{\prime}-1}},
\]
which simplifies to 
\[
\frac{1}{p^{\prime}}\log\beta-\frac{\E Z^{p^{\prime}}\log Z}{p^{\prime}\beta^{p^{\prime}-1}}.
\]
From the previous considerations we now have the desired formula for
the $p^{\prime}$-derivative of $\EVaR_{\alpha}^{p}(Y)$: 
\begin{equation}
\frac{\mathrm{d}}{\mathrm{d}p^{\prime}}\EVaR_{\alpha}^{p}(Y)=\beta^{\frac{p^{\prime}-1}{p^{\prime}}}\left\Vert \left(Y-x\right)_{+}\right\Vert _{\frac{p^{\prime}}{p^{\prime}-1}}\left(\frac{1}{p^{\prime}}\log\beta-\frac{\E Z^{p^{\prime}}\log Z}{p^{\prime}\beta^{p^{\prime}-1}}\right).\label{eq:7-5}
\end{equation}
It remains to be seen that this derivative is increasing in $p^{\prime}$,
which is equivalent to showing that 
\[
p^{\prime}\mapsto\beta^{\prime\frac{p^{\prime}-1}{p^{\prime}}}\left\Vert \left(Y-x\right)_{+}\right\Vert _{\frac{p^{\prime}}{p^{\prime}-1}}\left(\frac{\E Z^{p^{\prime}}\log Z}{p^{\prime}\beta^{p^{\prime}-1}}-\frac{1}{p^{\prime}}\log\beta\right)
\]
is decreasing. We first show that $p^{\prime}\mapsto\beta^{\prime\frac{p^{\prime}-1}{p^{\prime}}}\left\Vert \left(Y-x\right)_{+}\right\Vert _{\frac{p^{\prime}}{p^{\prime}-1}}$
is decreasing. From Theorem~\ref{thm:Monotonicity} we know that
$p^{\prime}\mapsto\EVaR_{\alpha}^{p}(Y)$ is decreasing and hence
for the optimizer $x$ of $\EVaR_{\alpha}^{p}(Y)$ we have
\[
\frac{\mathrm{d}}{\mathrm{d}p^{\prime}}\EVaR_{\alpha}^{p}(Y)=\frac{\mathrm{d}}{\mathrm{d}p^{\prime}}\beta^{\prime\frac{p^{\prime}-1}{p^{\prime}}}\left\Vert \left(Y-x\right)_{+}\right\Vert _{\frac{p^{\prime}}{p^{\prime}-1}}<0
\]
and thus $p^{\prime}\mapsto\beta^{\prime\frac{p^{\prime}-1}{p^{\prime}}}\left\Vert \left(Y-x\right)_{+}\right\Vert _{\frac{p^{\prime}}{p^{\prime}-1}}$
is decreasing. 

It remains to show that 
\[
p^{\prime}\mapsto\frac{\E Z^{p^{\prime}}\log Z}{p^{\prime}\beta^{p^{\prime}-1}}-\frac{1}{p^{\prime}}\log\beta
\]
is decreasing. Note that $\E Z^{*}(p^{\prime})^{p^{\prime}}\log Z^{*}(p^{\prime})\leq\E Z^{*}(p^{\prime})^{p^{\prime}+1}\leq\E C(Y)\cdot\left(Z^{*}(p^{\prime}+1)\right)^{p^{\prime}+1}=C(Y)\cdot\beta^{p^{\prime}}$where
$C(Y)\geq1$ is a constant only depending on $Y$. Then
\begin{align*}
\frac{\E Z^{p^{\prime}}\log Z}{p^{\prime}\beta^{p^{\prime}-1}}-\frac{1}{p^{\prime}}\log\beta & \leq\frac{C(Y)\beta^{p^{\prime}}-\beta^{p^{\prime}-1}\log\beta}{p^{\prime}\beta^{p^{\prime}-1}}\leq\frac{C(Y,\beta)}{p^{\prime}},
\end{align*}
where $C(Y,\beta)>0$ is a constant only depending on $Y$ and $\beta$
and thus we conclude that (cf.~\eqref{eq:7-5})
\[
\frac{\mathrm{d}}{\mathrm{d}p^{\prime}}\EVaR_{\alpha}^{p}(Y)=\beta^{\frac{p^{\prime}-1}{p^{\prime}}}\left\Vert \left(Y-x\right)_{+}\right\Vert _{\frac{p^{\prime}}{p^{\prime}-1}}\left(\frac{1}{p^{\prime}}\log\beta-\frac{\E Z^{p^{\prime}}\log Z}{p^{\prime}\beta^{p^{\prime}-1}}\right)
\]
is increasing in $p^{\prime}$, which is equivalent to the convexity
of $p^{\prime}\mapsto\EVaR_{\alpha}^{p}$.
\end{proof}
Before we give the proof of the main Theorem~\ref{thm:log-convex}
we analyze the optimizer of the infimum representation of $\EVaR$. 
\begin{lem}
Let $\alpha\in(0,1)$, $\beta=\frac{1}{1-\alpha}$ and $Y\in L^{\infty}$.
As usual for $p>1$ we set $p^{\prime}=\frac{p}{p-1}$. Consider the
infimum representation of $\EVaR_{\alpha}^{p}(Y)=\inf_{x\in\mathbb{R}}x+\beta^{\frac{1}{p}}\left\Vert \left(Y-x\right)_{+}\right\Vert _{p}$
with optimizer $x^{*}$. Then the mapping $p^{\prime}\mapsto x^{*}(p^{\prime})$
is nondecreasing. 
\end{lem}

\begin{proof}
We may restrict ourselves to $Y>0$ and recall that the derivative
of $\EVaR_{\alpha}^{p}(Y)$ with respect to $x$ is given by 
\[
1-\beta^{\frac{1}{p}}\left(\E\left(Y-x^{*}(p^{\prime})\right)_{+}^{p}\right)^{\frac{1}{p}-1}\E\left(Y-x^{*}(p^{\prime})\right)_{+}^{p-1}
\]
whenever the derivative exists. It follows that $p^{\prime}\mapsto x^{*}(p^{\prime})$
is nondecreasing in $p^{\prime}$ if and only if $1-\beta^{\frac{p^{\prime}-1}{p^{\prime}}}\left(\E\left(Y-x^{*}(p_{0}^{\prime})\right)_{+}^{\frac{p^{\prime}}{p^{\prime}-1}}\right)^{-\frac{1}{p^{\prime}}}\E\left(Y-x^{*}(p_{0}^{\prime})\right)_{+}^{\frac{1}{p^{\prime}-1}}<0$
for all $p^{\prime}>p_{0}^{\prime}$. By convexity of~\eqref{eq: EVaRinf1}
there is a unique optimal $x^{*}(p^{\prime})$ for any $p^{\prime}>1$.
Furthermore the optimizer $x^{*}(p^{\prime})$ is characterized by
\[
\beta\left(\frac{\left\Vert \left(Y-x\right)_{+}\right\Vert _{\frac{1}{p^{\prime}-1}}}{\left\Vert \left(Y-x\right)_{+}\right\Vert _{\frac{p^{\prime}}{p^{\prime}-1}}}\right)^{\frac{p^{\prime}}{(p^{\prime}-1)^{2}}}=1.
\]
It is enough to verify the that $p^{\prime}\mapsto\left(\frac{\left\Vert \left(Y-x\right)_{+}\right\Vert _{\frac{1}{p^{\prime}-1}}}{\left\Vert \left(Y-x\right)_{+}\right\Vert _{\frac{p^{\prime}}{p^{\prime}-1}}}\right)^{\frac{p^{\prime}}{(p^{\prime}-1)^{2}}}$
is nondecreasing for fixed $x<\esssup Y$. As the logarithm is a strictly
increasing function, we may consider the mapping $p^{\prime}\mapsto\log\left(\frac{\left\Vert \left(Y-x\right)_{+}\right\Vert _{\frac{1}{p^{\prime}-1}}}{\left\Vert \left(Y-x\right)_{+}\right\Vert _{\frac{p^{\prime}}{p^{\prime}-1}}}\right)^{\frac{p^{\prime}}{(p^{\prime}-1)^{2}}}$
with $p^{\prime}$-derivative
\[
\left(\frac{\mathrm{d}}{\mathrm{d}p^{\prime}}\frac{p^{\prime}}{(p^{\prime}-1)^{2}}\right)\cdot\log\frac{\left\Vert \left(Y-x\right)_{+}\right\Vert _{\frac{1}{p^{\prime}-1}}}{\left\Vert \left(Y-x\right)_{+}\right\Vert _{\frac{p^{\prime}}{p^{\prime}-1}}}+\frac{p^{\prime}}{(p^{\prime}-1)^{2}}\cdot\frac{\mathrm{d}}{\mathrm{d}p^{\prime}}\log\frac{\left\Vert \left(Y-x\right)_{+}\right\Vert _{\frac{1}{p^{\prime}-1}}}{\left\Vert \left(Y-x\right)_{+}\right\Vert _{\frac{p^{\prime}}{p^{\prime}-1}}}.
\]
By monotonicity of the $p$-Norms and by logarithmic convexity of
$p^{\prime}\mapsto\left\Vert \cdot\right\Vert _{p}$ we may conclude
that 
\[
\frac{\mathrm{d}}{\mathrm{d}p^{\prime}}\log\left(\frac{\left\Vert \left(Y-x\right)_{+}\right\Vert _{p-1}}{\left\Vert \left(Y-x\right)_{+}\right\Vert _{p}}\right)^{p(p-1)}>0.
\]
It follows immediately that $p^{\prime}\mapsto1-\beta^{\frac{1}{p}}\left(\E\left(Y-x\right)_{+}^{p}\right)^{\frac{1}{p}-1}\E\left(Y-x\right)_{+}^{p-1}$
is decreasing and thus 
\[
1-\beta^{\frac{1}{p}}\left(\E\left(Y-x^{*}(p_{0}^{\prime})\right)_{+}^{p}\right)^{\frac{1}{p}-1}\E\left(Y-x^{*}(p_{0}^{\prime})\right)_{+}^{p-1}<0
\]
for all $p^{\prime}>p_{0}^{\prime}$.
\end{proof}
The proof of Theorem~\ref{thm:log-convex} now follows as a simple
corollary of the preceding 2 lemmas.
\begin{proof}[Proof of Theorem~\ref{thm:log-convex}]
Suppose the assumptions of Lemma~\ref{lem:convexEVaR} are satisfied.
With out loss of generality we may assume that $Y>0$ and that $Y$
is not constant. We then consider the derivative with respect to $p^{\prime}$
using the envelope theorem and obtain
\begin{align*}
\frac{\mathrm{d}}{\mathrm{d}p^{\prime}}\log\EVaR_{\alpha}^{p}(Y) & =\frac{\beta^{\frac{p^{\prime}-1}{p^{\prime}}}\left\Vert \left(Y-x^{*}(p^{\prime})\right)_{+}\right\Vert _{\frac{p^{\prime}}{p^{\prime}-1}}\left(\frac{1}{p^{\prime}}\log\beta-\frac{\E Z^{p^{\prime}}\log Z}{p^{\prime}\beta^{p^{\prime}-1}}\right)}{x^{*}(p^{\prime})+\beta^{\frac{p^{\prime}-1}{p^{\prime}}}\left\Vert \left(Y-x^{*}(p^{\prime})\right)_{+}\right\Vert _{\frac{p^{\prime}}{p^{\prime}-1}}}.
\end{align*}
We define $C=x^{*}\cdot\left(\beta^{\frac{p^{\prime}-1}{p^{\prime}}}\left\Vert \left(Y-x^{*}\right)_{+}\right\Vert _{\frac{p^{\prime}}{p^{\prime}-1}}\right)^{-1}+1>0$,
from which it follows that 
\[
\frac{\mathrm{d}}{\mathrm{d}p^{\prime}}\log\EVaR_{\alpha}^{p}(Y)=\frac{\left(\frac{1}{p^{\prime}}\log\beta-\frac{\E Z^{p^{\prime}}\log Z}{p^{\prime}\beta^{p^{\prime}-1}}\right)}{C}.
\]
It remains to show that $p^{\prime}\mapsto C(p^{\prime})$ is increasing.
Since $p^{\prime}\mapsto x^{*}$ is increasing it suffices to show
that $x^{*}\mapsto C(x^{*})$ is increasing. To see this we differentiate
$C(x^{*}(p^{\prime}))$ with respect to $x$, 
\begin{align*}
\frac{\mathrm{d}}{\mathrm{d}x}C(x^{*}(p^{\prime})) & =\beta^{\frac{p^{\prime}-1}{p^{\prime}}}\left\Vert \left(Y-x^{*}\right)_{+}\right\Vert _{\frac{p^{\prime}}{p^{\prime}-1}}+x^{*}(p^{\prime})\cdot\left(\beta^{\frac{p^{\prime}-1}{p^{\prime}}}\left(\E\left(Y-x^{*}\right)_{+}^{p}\right)^{\frac{1}{p}-1}\E\left(Y-x^{*}\right)_{+}^{p-1}\right)\\
 & =\beta^{\frac{p^{\prime}-1}{p^{\prime}}}\left\Vert \left(Y-x^{*}\right)_{+}\right\Vert _{\frac{p^{\prime}}{p^{\prime}-1}}+x^{*}(p^{\prime})\geq0.
\end{align*}
We conclude that $p^{\prime}\mapsto C(Y,\alpha,p^{\prime})$ is increasing
and by the proof of Lemma~\ref{lem:convexEVaR} it now follows that
\[
p^{\prime}\mapsto\frac{\left(\frac{\E Z^{p^{\prime}}\log Z}{p^{\prime}\beta^{p^{\prime}-1}}-\frac{1}{p^{\prime}}\log\beta\right)}{C(Y,\alpha,p^{\prime})}
\]
is decreasing which concludes the assertion.
\end{proof}

\end{document}